\documentclass[11pt]{amsart}
\usepackage{fullpage}
\usepackage{amsmath}
\usepackage{ mathrsfs }
\usepackage{ dsfont }
\usepackage{indentfirst}
\usepackage{mathtools}
\usepackage{amsfonts}
\usepackage{amsthm}
\usepackage{graphicx}
\usepackage{float}
\usepackage[english]{babel}
\usepackage{fullpage}
\usepackage{hyperref}

\DeclareUnicodeCharacter{2212}{-}
\newcommand\blfootnote[1]{%
  \begingroup
  \renewcommand\thefootnote{}\footnote{#1}%
  \addtocounter{footnote}{-1}%
  \endgroup
}

\hypersetup{
    colorlinks,
    linkcolor={red!80!black},
    citecolor={blue!50!black},
    urlcolor={blue!80!black}
}

\usepackage{enumerate}
\usepackage{enumitem}

\def\caselabel{\textbf{\textit{Case~}\upshape(\!{\itshape \arabic*\,}\!)}}
\def\alabel{\upshape({\itshape \alph*\,})}
\def\rmlabel{\upshape({\itshape \roman*\,})}
\def\arlabel{\upshape({\itshape \arabic*\,})}

\usepackage{tasks}
\usepackage{color}
\usepackage{theoremref}
\usepackage{relsize}
\usepackage{mathdots}

\usepackage{ amssymb }
\usepackage{ textcomp }

\usepackage{tikz}
\usetikzlibrary{graphs,graphs.standard,quotes}
\tikzset{
vtx/.style={inner sep=2pt, outer sep=0pt, circle, fill=black,draw=black}
}

\usepackage{accents}

\newtheorem{theorem}{Theorem}[section]
\newtheorem{fact}[theorem]{Fact}
\newtheorem{lemma}[theorem]{Lemma}
\newtheorem{question}[theorem]{Question}
\newtheorem{proposition}[theorem]{Proposition}
\newtheorem{corollary}[theorem]{Corollary}

\newtheorem{claim}[theorem]{Claim}
\newtheorem{definition}[theorem]{Definition}

\newtheorem{prop}[theorem]{Proposition}

\newtheorem{remark}[theorem]{Remark}

\numberwithin{equation}{section}

\def\eps{\varepsilon}

\def\R{\mathbb{R}}

\usepackage{scalerel}
\def\Pk{P_k^{\scaleto{{\leqslant}}{4.3pt}}}
\def\Pkk{P_{k+1}^{\scaleto{{\leqslant}}{4.3pt}}}
\def\Pkkk{P_{k-1}^{\scaleto{{\leqslant}}{4.3pt}}}
\def\sco{$\star$-canonical ordering}
\def\sceo{$\star$-canonically edge-ordered}

\def\vmin{\underline{v}}
\def\vmax{\overline{v}}
\def\umin{\underline{u}}
\def\umax{\overline{u}}
\def\phi{\varphi}

\newcommand{\back}[1]{\accentset{\scaleto{\leftharpoonup}{3pt}\vspace{-2pt}}{#1}}

\def\qand{\quad\text{and}\quad}
\def\qqand{\qquad\text{and}\qquad}

\newcommand{\minfrom}[1]{v_{#1}v_{\min}}

\newcommand{\mintill}[1]{v_{\min}v_{#1}}
\newcommand{\maxtill}[1]{v_{\max}v_{{#1}}}

\makeatletter
\newcommand*{\rom}[1]{\expandafter\@slowromancap\romannumeral #1@}
\makeatother

\newenvironment{claimproof}[1]{\par\noindent\textit{Proof of the claim:}\space#1}{\leavevmode\unskip\penalty9999 \hbox{}\nobreak\hfill\quad\hbox{$\blacksquare$}}

\def\COMMENT#1{}
\let\COMMENT=\footnote

\allowdisplaybreaks

\title{Tiling edge-ordered graphs with monotone paths and other structures}

\author{
    Igor Araujo, Sim\'on Piga, Andrew Treglown and Zimu Xiang
    }

\thanks{IA: Department of Mathematics, University of Illinois at Urbana-Champaign, Urbana, Illinois 61801, USA. Email: \texttt{igoraa2@illinois.edu}. Research partially supported by UIUC Campus Research Board RB 22000.
\\ \indent SP: University of Birmingham, United Kingdom. Email: \texttt{s.piga@bham.ac.uk}. Research supported by EPSRC grant EP/V002279/1.
\\ \indent AT: University of Birmingham, United Kingdom. Email: \texttt{a.c.treglown@bham.ac.uk}. Research supported by EPSRC grant EP/V002279/1.
\\ \indent ZX: Department of Mathematics, University of Illinois at Urbana-Champaign, Urbana, Illinois 61801, USA. Email: \texttt{zimux2@illinois.edu}.
}

\date{}

\begin{document}
 \maketitle
\begin{abstract}
Given graphs $F$ and $G$, a perfect $F$-tiling in $G$ is a collection of vertex-disjoint copies of $F$ in $G$ that together cover all the vertices in $G$. The study of the minimum degree threshold forcing a perfect $F$-tiling in a graph $G$ has a long history, culminating in the K\"uhn--Osthus theorem [Combinatorica 2009] which resolves this problem, up to an additive constant, for all graphs $F$.
In this paper we initiate the study of the analogous question for edge-ordered graphs. 
In particular, we characterize for which edge-ordered graphs $F$ this problem is well-defined. We also apply the absorbing method to asymptotically determine the minimum degree threshold for forcing a perfect $P$-tiling in an edge-ordered graph, where $P$ is any fixed monotone path.
\end{abstract}

\blfootnote{The main results of this paper were first announced in the conference abstract~\cite{conf}.}

\section{Introduction}
\subsection{Monotone paths in edge-ordered graphs}
An \emph{edge-ordered  graph} $G$ is a graph equipped with a  total order $\leq$ of its edge set $E(G)$. 
Usually we will think of a total order of $E(G)$ as a labeling of the edges with labels from $\mathbb R$, where the labels inherit the total order of $\mathbb R$ and where edges are assigned distinct labels. A path $P$ in $G$ is 
\emph{monotone} if the consecutive edges of $P$ form a monotone sequence with respect to $\leq$.
We write $\Pk$ for the monotone path of length $k$ (i.e., on $k$ edges).

Let $F$ and $G$ be edge-ordered graphs. We say that $G$ \emph{contains} $F$ if $F$ is isomorphic to a subgraph $F'$ of $G$; here, crucially, the total order of $E(F)$ must be the same as the total order of $E(F')$ that is inherited from the total order of $E(G)$. In this case we say $F'$ is a \emph{copy of~$F$ in~$G$}.
For example, if $G$ contains a path $F'$ of length $3$ with consecutive edges labeled $5$, $17$ and $4$ then  $F'$ is a copy of the path $F$ of length $3$ with consecutive edges labeled  $2$, $3$ and $1$.

The study of monotone paths in edge-ordered graphs dates back to the 1970s. Chv\'atal and Koml\'os~\cite{chvkom} raised the following question: what is the largest integer $f(K_n)$ such that every edge-ordering of $K_n$ contains a copy of the monotone path $P_{f(K_n)}^{\scaleto{{\leqslant}}{4.3pt}}$?  Over the years there have been several papers on this topic~\cite{bkwpstw, burger, ccs, gk, milans, rodl}. In a recent breakthrough, Buci\'c, Kwan,  Pokrovskiy, Sudakov,  Tran, and  Wagner~\cite{bkwpstw} proved that $f(K_n)\geq n^{1-o(1)}$. The best known upper bound on $f(K_n)$ is due to
Calderbank,  Chung, and  Sturtevant~\cite{ccs} who proved that $f(K_n)\leq (1/2+o(1))n$. There have also been numerous papers on the wider question of 
the largest integer $f(G)$ such that every edge-ordering of a graph $G$ contains a copy of a monotone path  of length $f(G)$. See the introduction of~\cite{bkwpstw} for a detailed overview of the related literature.

A classical result of R\"odl~\cite{rodl} yields a Tur\'an-type result for monotone paths: every edge-ordered graph with $n$ vertices and with at least~$k(k+1)n/2$ edges
contains a copy of $\Pk$. 
More recently, Gerbner, Methuku, Nagy, P\'alv\"olgyi, Tardos, and Vizer~\cite{gmnptv} initiated the systematic study of the Tur\'an problem for edge-ordered graphs.

It is also natural to seek conditions that force an edge-ordered graph $G$ to contain a collection of vertex-disjoint monotone paths $\Pk$ that cover all the vertices in $G$, that is, a \emph{perfect $\Pk$-tiling} in~$G$. Our first result asymptotically determines the minimum degree  threshold that forces a {perfect $\Pk$-tiling}.
\begin{theorem}\label{Pkfactor}
Given any $k \in \mathbb N$ and $\eta >0$, there exists an $n_0 \in \mathbb N$ such that if $n \geq n_0$ where  $(k+1) | n$ then the following holds:
if $G$ is an $n$-vertex edge-ordered graph with minimum degree 
$$ \delta(G)\geq (1/2+\eta)n$$
then $G$ contains a perfect $\Pk$-tiling. 
Moreover, for all $n\in \mathbb N$ with $(k+1)\vert n$, there is an $n$-vertex edge-ordered graph $G_0$ with $\delta(G_0)\geq \lfloor n/2\rfloor-2$ that does not contain a perfect $\Pk$-tiling.
\end{theorem}
The proof of Theorem~\ref{Pkfactor} provides the first application of the so-called \emph{absorbing method} in
the setting of edge-ordered graphs.

\subsection{The general problem}
Given edge-ordered graphs $F$ and $G$, an \emph{$F$-tiling} in $G$ is a collection of vertex-disjoint copies of $F$ in $G$; an $F$-tiling in $G$ is \emph{perfect} if it covers all the vertices in $G$.
In light of Theorem~\ref{Pkfactor} we raise the following general question.
\begin{question}\label{ques1}
Let $F$ be a fixed edge-ordered graph on~$f\in \mathbb N$ vertices and let $n \in \mathbb N$ be divisible by $f$.
What is the smallest integer $f(n,F)$ such that every edge-ordered graph on $n$ vertices and of minimum degree at least $f(n,F)$ contains a perfect $F$-tiling?
\end{question}
Theorem~\ref{Pkfactor} implies that $f(n,\Pk)=(1/2+o(1))n$ for all $k\in \mathbb N$.
Note that the \emph{unordered} version of Question~\ref{ques1} had been well-studied since the 1960s
(see, e.g.,~\cite{alonyuster, cor, hs, kssAY,  kuhn2}) and forty-five years later 
 a complete solution, up to an additive constant term, was obtained via a theorem of K\"uhn and Osthus~\cite{kuhn2}.
Very recently, the \emph{vertex-ordered graph} version of this problem has been asymptotically resolved~\cite{blt, andrea}.

Question~\ref{ques1} has a rather different flavor to its graph and vertex-ordered graph counterparts. 
In particular, there are edge-ordered graphs $F$ for which, given \emph{any} $n \in \mathbb N$, there exists an edge-ordering $\leq$ of the complete graph $K_n$ that does not contain a copy of $F$.\footnote{See~\cite{gmnptv} for various examples of such $F$.} 
Thus, for such $F$, Question~\ref{ques1} is trivial in the sense that clearly there is no minimum degree threshold $f(n,F)$ for forcing a perfect $F$-tiling.
This motivates Definitions~\ref{def:Turanable} and~\ref{def:tile} below.

\begin{definition}[Tur\'anable]\label{def:Turanable}\rm
An edge-ordered graph $F$ is \emph{Tur\'anable} if there exists a $t\in \mathbb N$ such every edge-ordering of the graph $K_t$ contains a copy of $F$. 
\end{definition}
An unpublished result of Leeb (see, e.g.,~\cite{gmnptv, ner}) characterizes all those edge-ordered graphs $F$ that are Tur\'anable.
Moreover, a result of Gerbner, Methuku, Nagy, P\'alv\"olgyi, Tardos, and Vizer~\cite[Theorem 2.3]{gmnptv}
shows that  the so-called \emph{order chromatic number} is the parameter that governs the Tur\'an threshold for 
Tur\'anable edge-ordered graphs $F$.

\begin{definition}[Tileable] \label{def:tile} \rm
An edge-ordered graph $F$ on $f$ vertices is \emph{tileable} if there exists a $t\in \mathbb N$ divisible by $f$ such that every edge-ordering of the graph $K_t$  contains a perfect $F$-tiling.
\end{definition}
Let $F$ be a tileable edge-ordered graph on $f$ vertices and let $T(F)$ be the smallest possible choice of $t \in \mathbb N$ in  Definition~\ref{def:tile} for $F$.
It is easy to see that every edge-ordering of the graph $K_s$  contains a perfect $F$-tiling for every $s \geq T(F)$ that is divisible by $f$.
Note that Theorem~\ref{Pkfactor} implies that $\Pk$ is tileable for all $k \in \mathbb N$.

The second objective of this paper is to provide a characterization of those edge-ordered graphs that are tileable; see Theorem~\ref{thm:character}. Thus, this characterizes for which edge-ordered graphs $F$ Question~\ref{ques1} is well-defined. The precise characterization of the tileable edge-ordered graphs is a little involved, and depends on twenty edge-orderings of $K_f$; as such, we defer the statement of Theorem~\ref{thm:character} to Section~\ref{subsec:state}.

In Section~\ref{subsec:examples} we highlight some interesting  properties of the class of tileable edge-ordered graphs. We show that there are infinitely many 
 Tur\'anable edge-ordered graphs that are not tileable; see Proposition~\ref{prop::Dn}. 
In \cite{gmnptv} it is proven that no edge-ordering of $K_4$ is Tur\'anable and consequently, any edge-ordered graph containing a copy of $K_4$ is not Tur\'anable and therefore not tileable. Thus, for an edge-ordered graph to be tileable it cannot be too `dense'.
Here we prove that no edge-ordering of~$K_4^-$ is tileable\footnote{Recall that $K_t^-$ denotes the graph obtained from $K_t$ by removing an edge.}; see~Proposition~\ref{prop::K4-}.
However, we prove that the property of being tileable is not closed under subgraphs and there are in fact connected tileable edge-ordered graphs that contain copies of~$K_4^-$ (see Corollary~\ref{cor:K_4-}). 

Another curious property is exhibited by  \emph{monotone cycles}.
We say that an edge-ordered cycle~$C_n$ with $V(C_n)=\{u_1,\dots, u_n\}$ is \emph{monotone} if the edges are ordered as $u_1u_2<u_2u_3<\dots<u_{n-1}u_n<u_nu_1$. 
We prove that monotone cycles of odd length are tileable whilst monotone cycles of even length are not
Tur\'anable; see Propositions~\ref{prop::monotone_odd} and~\ref{lem::monotone_even}.



\smallskip 

 A graph $H$ is \emph{universally tileable} if for any given ordering of~$E(H)$, the resulting edge-ordered graph is tileable. 
 Similarly, we say that $H$ is \emph{universally Tur\'anable} if given any edge-ordering of $H$, the resulting edge-ordered graph is Tur\'anable. 
Using~\cite[Theorem~2.18]{gmnptv} it is easy to characterize those  graphs $H$ that are universally tileable.

\begin{theorem}\label{thm:uni}
Let  $H$ be a graph. The following are equivalent:
\begin{enumerate}[wide, leftmargin=23pt, labelindent=3pt, label=\alabel]
    \item $H$ is universally tileable; \label{it:univtil}
    \item $H$ is universally Tur\'anable; \label{it:univturan}
    \item \label{it:univdescrip}
    \begin{enumerate}[wide,leftmargin=15pt, labelindent=0pt,  label=\rmlabel]
        \item $H$ is a star forest (possibly with isolated vertices),\footnote{A \emph{star forest} is a graph whose components are all stars.} or
        \item  $H$ is a path on three edges together with a (possibly empty) collection of isolated vertices,~or 
        \item $H$ is a copy of $K_3$ together with
    a (possibly empty) collection of isolated vertices. 
    \end{enumerate}
\end{enumerate}
\end{theorem}

In Section~\ref{subsec:char2} we  determine the asymptotic value of $f(n,F)$ for all connected universally tileable edge-ordered graphs $F$.

\smallskip

Our characterization of tileable edge-ordered graphs lays the ground for the systematic study of Question~\ref{ques1}. The second and third authors will investigate this problem further in a forthcoming paper. Already though we can say something about this question. 
Indeed, an almost immediate consequence of the Hajnal--Szemer\'edi theorem~\cite{hs} is the following result.
\begin{theorem}\label{hscorollary}
Let $F$ be a tileable edge-ordered graph and let $T(F)$ be the smallest possible choice of $t \in \mathbb N$ in  Definition~\ref{def:tile} for $F$.
Given any  integer $n\geq T(F)$ divisible by $|F|$,
$$f(n,F) \leq \left ( 1- \frac{1}{T(F)} \right ) n.$$
\end{theorem}

The paper is organized as follows. In Section~\ref{subsec:state} we state the characterization of all tileable edge-ordered graphs (Theorem~\ref{thm:character}). 
Then, in Section~\ref{subsec:examples} we use this theorem to provide some basic properties of the family of tileable edge-ordered graphs and some general examples.
We give the proof of Theorem~\ref{thm:character} in 
 Section~\ref{subsec:proof}.
 In Section~\ref{subsec:char2} we consider universally tileable graphs, and give the proof of Theorem~\ref{thm:uni}.
 The proof of Theorem~\ref{hscorollary} is given in Section~\ref{subsec:hscor}.
 In Section~\ref{sec:mainproof} we give the
  proof of Theorem~\ref{Pkfactor}.
Finally,  some concluding remarks are made in Section~\ref{sec:conc}.

\smallskip

\subsection*{Notation}
Let $G$ be an (edge-ordered) graph. 
We write $V(G)$ and $E(G)$ for its vertex and edge sets respectively. 
We denote an edge~$\{u,v\}\in E(G)$ by~$uv$, omitting parenthesis and commas.
Define $|G|:=|V(G)|$. 
Given some $X \subseteq V(G),$ we write $G[X]$ for the induced (edge-ordered) subgraph of $G$ with vertex set $X$.
Define $G\setminus X:=G[V(G) \setminus X]$.
Given $x \in V(G)$ we define $G-x:=G[V(G)\setminus\{x\}]$.  
We define $N_G(x)$ be the set of vertices adjacent to $x$ in $G$ and set
$d_G(x):=|N_G(x)|$. When the graph $G$ is clear from the context, we will omit the subscript $G$ here. 
We say an edge $e_1$ in $G$ is \emph{larger} than another edge $e_2$ if $e_2$ occurs before $e_1$ in the total order of $E(G)$; in this case we may write $e_1>e_2$ or $e_2<e_1$.
We define \textit{smaller} analogously.
A sequence $\{ e_i \} _{i \in [t] } $ of edges is \emph{monotone} if $e_1<e_2< \dots < e_t$ or
$e_1>e_2> \dots > e_t$.

Given an (unordered) graph $G$ we write $G^{\scaleto{{\leqslant}}{4.3pt}}$ to denote an edge-ordered graph obtained from $G$ by equipping $E(G)$ with a total order~$\leqslant$. 
We say that~$G$ is the \emph{underlying graph of~$G^{\scaleto{{\leqslant}}{4.3pt}}$}.
Given a graph $G$ together with an (injective) labeling $L: E(G) \to \R$ of its edges, we define the \emph{edge-ordering induced by the labeling $L$} so that $e_i<e_j$ if and only if $L(e_i)<L(e_j)$. As such, $L$ gives rise to an edge-ordered graph. Note that two \emph{different}  labelings can give rise to the \emph{same} edge-ordered graph. For example, a path whose edges are labeled $1$, $2$, and $3$ respectively is a monotone path; likewise a path whose edges are labeled $1$, $e$, and $\pi$ respectively is a monotone path.

We denote the (unordered) path of length~$k$ by~$P_k$ and sometimes we identify a copy of $P_k$ with its sequence of vertices~$v_1\cdots v_{k+1}$ where~$v_iv_{i+1}\in E(P_k)$ for all~$i\in [k]$.
Given distinct $a_1,\dots , a_t \in \mathbb R$ we write $a_1\dots a_t$ for the edge-ordered path on $t$ edges whose $i$th edge has label $a_i$.
For example, $P=132$ is the edge-ordered path on four vertices $v_1, v_2, v_3, v_4$ whose first edge $v_1v_2$ is labeled $1$, second edge $v_2v_3$ is labeled $3$, and third edge $v_3v_4$ is labeled $2$. 

Given $k \in \mathbb N$ and a set $X$, we write $\binom{X}{k}$
for the collection of all subsets of $X$ of size $k$.

\section{The characterization of all tileable edge-ordered graphs}\label{sec:character}

\subsection{The characterization theorem}\label{subsec:state}


The following Ramsey-type result, attributed to Leeb (see~\cite{gmnptv, ner}), says that in every sufficiently large edge-ordered complete graph we  must always find a subgraph which is~\emph{canonically ordered} (see Definition~\ref{def:canonical}).
Before giving the precise description of the canonical orderings, let us present Leeb's result.

\begin{proposition}\label{prop:canonical}
For every~$k\in \mathbb N$ there is an~$m\in \mathbb N$ such that every edge-ordered complete graph~$K_m$ contains a copy of~$K_k$ that is canonically edge-ordered with respect to some ordering of the vertices. \qed
\end{proposition}




We now define the canonical orderings of~$K_n$.

\begin{definition}\label{def:canonical} \rm 
Given~$n\in \mathbb N,$ we denote by $\{v_1,\dots,v_n\}$ the vertex set of the complete graph $K_n$. 
The following labelings~$L_1$, $L_2$, $L_3$, and $L_4$ induce the \emph{canonical orderings} of $K_n$.
\medskip 

$\bullet$ \emph{min ordering}: For $1\le i<j\le n$ the label of the edge 
$v_iv_j$ is $L_1(v_iv_j)= 2ni+j-1$.

\smallskip 

$\bullet$ \emph{max ordering}: For $1\le i<j\le n$ the label of the edge 
$v_iv_j$ is $L_2(v_iv_j)=(2n-1)j+i$.

\smallskip 

$\bullet$ \emph{inverse min ordering}: For $1\le i<j\le n$ the label of the edge 
$v_iv_j$ is $L_3(v_iv_j)=(2n+1)i-j$.

\smallskip 

$\bullet$ \emph{inverse max ordering}: For $1\le i<j\le n$ the label of the edge 
$v_iv_j$ is $L_4(v_iv_j)=2nj-i+n$.

\smallskip \smallskip 

We say that min, max, inverse min, and inverse max are \emph{types} of canonical orderings and that the labelings~$L_1$, $L_2$, $L_3$, and $L_4$ are the \textit{standard labelings} for those types.\footnote{The labelings~$L_1$, $L_2$, $L_3$, and~$L_4$ presented here differ from those used in \cite{gmnptv}. 
However, the induced edge-orderings are the same. 
This labeling will be useful for Definition~\ref{def:starcanonical}.} 
To emphasize, in the statement of Proposition~\ref{prop:canonical}, by `a copy of~$K_k$ that is canonically edge-ordered', we mean that the edge-ordering of $K_k$ is the same as the edge-ordering induced by the labeling $L_i$, for some $i \in [4]$.
\end{definition}

Observe that the max and inverse max orderings are the `reverse' of the min and inverse min orderings respectively. For example, if you reverse the total order of $E(K_n)$ induced by the min ordering $L_1$, then you obtain an edge-ordered graph whose total order  is now induced by the max ordering $L_2$; here though vertex $v_n$ is playing the role of $v_1$, $v_{n-1}$ is playing the role of $v_2$, etc.

\begin{remark}\label{rem:con}\rm
Whilst the standard labelings  formally define the canonical orderings,  recalling the following intuitive explanations of these orderings will aid the reader throughout the paper:
\begin{itemize}
    \item \emph{min ordering}: the smallest edges are incident to $v_1$ so that $v_1v_2<\dots <v_1v_n$; the next smallest edges are those that go from $v_2$ to the `right' of $v_2$ so that $v_2v_3<\dots <v_2v_n$; the next smallest 
  edges are those that go from $v_3$ to the `right' of $v_3$ so that $v_3v_4<\dots <v_3v_n$, and so forth.
    \item  \emph{max ordering}: the largest edges are incident to $v_n$ so that $v_1v_n<\dots <v_{n-1}v_n$; the next largest edges are those that go from $v_{n-1}$ to the `left' of $v_{n-1}$ so that $v_1v_{n-1}<\dots <v_{n-2}v_{n-1}$, and so forth.
    \item \emph{inverse min ordering}: the smallest edges are incident to $v_1$ so that $v_1v_n<\dots <v_1v_2$; the next smallest edges are those that go from $v_2$ to the `right' of $v_2$ so that $v_2v_n<\dots < v_2v_3$, and so forth.
     \item  \emph{inverse max ordering}: the largest edges are incident to $v_n$ so that $v_1v_n> \dots >v_{n-1}v_n$; the next largest edges are those that go from $v_{n-1}$ to the `left' of $v_{n-1}$ so that $v_1v_{n-1}>\dots >v_{n-2}v_{n-1}$, and so forth.
\end{itemize}
\end{remark}

In~\cite{gmnptv} it was observed that Proposition~\ref{prop:canonical} yields a full characterization of Tur\'anable graphs. 

\begin{theorem}[Tur\'anable characterization] \label{thm:turanable}  
An edge-ordered graph $F$ on $f$ vertices is Tur\'anable if and only if all  four canonical edge-orderings of $K_f$ contain a copy of $F$. \qed
\end{theorem}

In~\cite[Theorem 2.5]{gmnptv} they present a `family' version of Theorem~\ref{thm:turanable}, which implies that~$F$ is Tur\'anable if and only if~$F$ is contained in every canonical edge-ordering of~$K_n$, \textit{for all~$n\in \mathbb N$}. 
However, Theorem~\ref{thm:turanable} can be deduced easily from the following fact. 

\begin{fact}\label{fact:selfidentical}
Suppose~$k\leq n$ are positive integers. 
If~$K_n$ is canonically edge-ordered, then~$K_k\subseteq K_n$ is canonically edge-ordered. 
Moreover, $K_k$ has the same type of canonical edge-ordering as~$K_n$.\qed
\end{fact}

The picture is slightly different when one seeks a perfect $F$-tiling instead of just a single copy of $F$. To illustrate, consider a canonical ordering of $K_n$ with an extra `defective' vertex $x$, whose edges incident to it can have an arbitrary ordering.
To have a perfect $F$-tiling in this edge-ordered graph, there must be a copy of $F$ containing the vertex $x$.
This leads to a generalization of the canonical orderings above, which we call \emph{\sco{}s} (see Definition~\ref{def:starcanonical}). 
We obtain a similar characterization for tileable graphs as follows.
\begin{theorem}[Tileable characterization]\label{thm:character}
    An edge-ordered graph $F$ on $f$ vertices is tileable if and only if all  twenty \sco s of $K_f$ contain a copy of $F$.
\end{theorem}

To define the \sco s we will consider an edge-ordering of the complete graph~$K_{n+1}$ for which there is a vertex~$x\in V(K_{n+1})$ such that~$K_{n+1}-x$ is canonically ordered. 
Depending on the type of canonical ordering and the ordering of the edges incident to~$x$ we have, for all~$n\geq 4$,~twenty possible \sco s of~$K_{n+1}$.

\begin{definition}\label{def:starcanonical} \rm
Let $\{x, v_1,\dots,v_n\}$ denote the vertex set of $K_{n+1}$. 
Suppose~$L:E(K_{n+1})\to \R$ is a labeling of the edges of~$K_{n+1}$ such that its restriction to~$K_{n+1}-x$ is canonical with one of the standard labelings $L_1$, $L_2$, $L_3$, or $L_4$. 
Moreover, suppose that the labels~$x_i:=L(xv_i)$ for~$i\in [n]$ satisfy one of the following:

\smallskip \smallskip 

$\bullet$ \emph{Larger increasing orderings}: $x_n > \dots > x_2 > x_1 > \max\limits_{i<j}\{L(v_iv_j)\}$.

\smallskip 

$\bullet$ \emph{Larger decreasing orderings}: $x_1 > x_2 > \dots > x_n > \max\limits_{i<j}\{L(v_iv_j)\}$.

\smallskip 

$\bullet$ \emph{Smaller increasing orderings}: $x_1 < x_2 < \dots < x_n < \min\limits_{i<j}\{L(v_iv_j)\}$.

\smallskip

$\bullet$ \emph{Smaller decreasing orderings}: $x_n < \dots < x_2 < x_1 < \min\limits_{i<j}\{L(v_iv_j)\}$.

\smallskip 

$\bullet$ \emph{Middle increasing orderings}: $x_i = 2ni$ for all~$i\in [n]$.





\smallskip\smallskip 
Then, $L$ induces a~\emph{\sco{}} of~$K_{n+1}$.
We refer to the vertex~$x$ as \textit{the special vertex}.
\end{definition}

Observe that depending on the type of canonical ordering of $K_{n+1}-x$ there are four possible larger increasing orderings, larger decreasing orderings, smaller increasing orderings, smaller decreasing orderings and middle increasing orderings, respectively. 
We will refer to these twenty possible cases as \textit{types} of \sco s. 
Moreover, we will say that~$K_{n+1}-x$ is the~\emph{canonical part} of the \sco.
We sometimes refer to the eight  smaller increasing/decreasing orderings as the  \emph{smaller orderings}. We define the \emph{larger orderings},
\emph{increasing orderings},  and \emph{decreasing orderings}  analogously.

\begin{remark}\label{rem:middle}\rm
In contrast with the other types, in the four middle increasing orderings, the edges incident to the special vertex $x$ are `in between' the edges of the canonical ordering of~$K_{n+1}-x$. 
More precisely, we have:  
\begin{itemize}[leftmargin=0.9cm]
    \item If $K_{n+1}\!-x$ is a min ordering then $v_{i-1}v_n < xv_i <  v_iv_{i+1}$ for every $2\le i\le n-1$. Additionally, $xv_1<v_1v_2$ and $v_{n-1}v_n<xv_n$.
    \item If $K_{n+1}\!-x$ is a max ordering then $v_{i-1}v_i < xv_i <  v_1v_{i+1}$ for every $2\le i\le n-1$. Additionally, $xv_1<v_1v_2$ and $v_{n-1}v_n<xv_n$.
    \item If $K_{n+1}\!-\!x$ is an inverse min ordering then $v_{i}v_{i+1} < xv_i <  v_{i+1}v_{n}$ for every $1\le i\le n-2.$ Additionally, $v_{n-1}v_n<xv_{n-1}<xv_n$.
    \item If $K_{n+1}\!-\!x$ is an inverse max ordering then $v_{1}v_{i-1} < xv_i<  v_{i-1}v_{i}$ for every $3\le i\le n$. Additionally, $xv_1<xv_2<v_1v_2$.
\end{itemize}
\end{remark}

It is not hard to check that canonical orderings are \sco s.
In particular, a min ordering is a smaller increasing ordering, a max ordering is a larger increasing ordering, an inverse min ordering is a smaller decreasing ordering, and an inverse max ordering is a larger decreasing ordering.
In each case, the special vertex~$x$ plays the role of either the first or the last vertex in the canonical ordering.

\smallskip

The proof of the `forwards direction' of Theorem~\ref{thm:character} relies on the following fact for \sco s, analogous to Fact~\ref{fact:selfidentical} for canonical orderings.

\begin{fact}\label{fact:selfidentical_star}
Suppose~$k\leq n$ are positive integers. 
If~$K_{n+1}$ is \sceo{} with special vertex $x$, then every subgraph~$K_k\subseteq K_{n+1}$ with~$x\in V(K_k)$ is \sceo{} with the same type as~$K_{n+1}$.\footnote{Note that it follows from Fact~\ref{fact:selfidentical} that every subgraph $K_k\subseteq K_{n+1}$ with~$x\notin V(K_k)$ is canonically ordered of the same type as~$K_{n+1}-x$.} \qed
\end{fact}

The forwards direction of Theorem~\ref{thm:character} follows easily from this fact. Indeed, if $F$ is tileable, by definition there is some $n \in \mathbb N$ so that 
in any \sco{} of $K_{n+1}$ there is a perfect $F$-tiling.
Fact~\ref{fact:selfidentical_star} implies that
 in such a perfect $F$-tiling there is a copy $F'$ of $F$ which covers $x$ and  where~$K_{n+1}[V(F')]$ is \sceo{} with the same type as~$K_{n+1}$.  Thus, this implies that every \sco{} of $K_{f}$ contains a copy of $F$.

The proof of the backwards direction of Theorem~\ref{thm:character} makes use of an approach  analogous to that of Caro~\cite{caro}.
More precisely, the intuition is as follows. Choose $t\in \mathbb N$ to be sufficiently large compared to  $f$.
Recall that due to Proposition~\ref{prop:canonical}, in any edge-ordering of a sufficiently large  $K_{n_0}$ one must find a canonical copy of $K_t$. 
Now consider any edge-ordering  of $K_{n}$ where~$n$ is much larger than $n_0$.
We may repeatedly find vertex-disjoint copies of a canonical copy of $K_t$ in $K_{n}$ until we have fewer than $n_0$ vertices remaining.
That is, we have tiled the vast majority of~$K_n$ with canonical copies of $K_t$.
The idea is now to incorporate the currently uncovered vertices into these canonical $K_t$ and then split each such `tile' into many \sceo{} copies of $K_f$. Therefore, the resulting substructure in $K_n$ is a perfect tiling of \sceo{} copies of $K_f$. 
Now by the choice of $F$, each such copy of $K_f$ contains a spanning copy of $F$. Thus, $K_n$ contains a perfect $F$-tiling, as desired.

\smallskip

We defer the formal proof of Theorem~\ref{thm:character} to Section~\ref{subsec:proof}.
In the following subsection we will see some applications of Theorems~\ref{thm:turanable} and \ref{thm:character} to study some properties of the families of Tur\'anable and tileable graphs. In particular, in Proposition~\ref{prop::Dn} we apply Theorem~\ref{thm:character} to prove that the notions of tileable and Tur\'anable are genuinely different. 
More precisely, we provide an infinite family  of Tur\'anable edge-ordered graphs that are not tileable.

\subsection{Tur\'anable and tileable graphs}\label{subsec:examples}


Given an edge-ordered graph~$F$ we define the \emph{reverse of~$F$}, denoted by~$\back{F}$, as the same graph but in which all relations in the total order of the edges of~$F$ are reversed. 
More precisely, for~$F=(V,E)$ we have~$\back{F}=(V,E)$ and for every~$e_1,e_2\in E$ we have 
$e_1\leq_{\back{F}} e_2$
 if and only if~$e_2\leq_F e_1$, where~$\leq_F$ and~$\leq_{\back{F}}$ are the total orders of~$F$ and~$\back{F}$ respectively.
It is easy to see that~$F$ is Tur\'anable if and only if~$\back{F}$ is Tur\'anable. 
Indeed, let~$F$ be a Tur\'anable edge-ordered graph and consider any edge-ordered copy of~$K_t$, where~$t\in \mathbb N$ is given by Definition~\ref{def:Turanable}. 
Then $\back{K_t}$ contains a copy of~$F$, and hence,~$K_t$ contains a copy of~$\back{F}$; thus, $\back{F}$ is Tur\'anable.
The same argument shows that~$F$ is tileable if and only if~$\back{F}$ is tileable.

Throughout this subsection $v_i$  will  denote the $i$th vertex in a canonical ordering and~$x$ will  denote the special vertex of a \sco. 
Given edge-ordered graphs~$F$ and~$H$, we say that a map $\varphi: V(F) \longrightarrow V(H)$ is an \emph{embedding of~$F$ into $H$} if and only if 
\begin{itemize}
    \item $\varphi$ is injective, 
    \item for every edge~$uv\in E(F)$ we have~$\varphi(u)\varphi(v)\in E(H)$, and
    \item for every two edges~$uv, wz\in E(F)$ such that~$uv<wz$ in the total order of~$E(F)$, we have~$\varphi(u)\varphi(v) < \varphi(w)\varphi(z)$ in the total order of~$E(H)$. 
\end{itemize}
Observe that the fact that~$H$ contains a copy of $F$ means there is an embedding from $F$ into $H$. 
When the embedding $\varphi$ is clear from the context we do not explicitly state it, and we simply write~$u\mapsto v$ instead of~$\varphi(u) = v$. 

\smallskip

We now present a Tur\'anable graph that is not tileable. 
Consider the edge-ordered graph $D_n$ defined in \cite{gmnptv} as a graph on vertices $u_1,\dots,u_n$ containing all edges incident to $u_1$ or $u_n$. The edges are ordered as $u_1u_2<u_1u_3<\dots<u_1u_n<u_2u_n<\dots < u_{n-1}u_n$.

\begin{proposition} \label{prop::Dn}
Let $n \geq 4$. Then
$D_n$ is Tur\'anable but is not tileable.
\end{proposition}  

\begin{proof}
The fact that~$D_n$ is Tur\'anable for every~$n\geq 4$ was proven in~\cite[Proposition~2.12]{gmnptv}, so we only need to show that it is not tileable. 

We prove it is impossible to embed~$D_n$ into a \sceo{} $K_n$ of type larger decreasing whose canonical part is a min ordering.
Let~$\{x, v_1, \dots, v_{n-1}\}$ be the vertices of such a~\sco{} of~$K_n$ with special vertex $x$.
Assume for a contradiction that there is an embedding of~$D_n$ into this edge-ordered $K_n$.
Suppose first that the vertex~$u_1$ is embedded onto the special vertex~$x$.
Then, there are vertices~$v_i,v_j\in V(K_n)$ such that in our embedding we have
$$u_k \mapsto v_i \qqand u_n\mapsto v_j\,,$$
for some $k \in [n-1]\setminus\{1\}$.
This immediately yields a contradiction since $u_1u_k < u_ku_n$ in~$D_n$ whilst in this type of \sco{} $x v_i > v_iv_j$ for every distinct~$i,j\in [n-1]$.

Suppose now that $u_i$ is embedded onto the special vertex~$x$ where~$i\in[n-1]\setminus \{1\}$.
Then there are vertices~$v_j,v_k,v_\ell\in V(K_n)$ such that 
$$u_1 \mapsto v_j\,, \qquad u_m \mapsto v_k\,, \qqand u_n\mapsto v_\ell\,,$$
for some~$m\in [n-1]\setminus \{1\}$. 
Similarly to before, this yields a contradiction because $u_1u_i < u_mu_n$ while in this type of \sco{} we have $xv_j > v_kv_\ell$ for every distinct $j,k, \ell\in [n-1]$.

The only remaining case is when~$u_n$ is embedded onto the special vertex~$x$. 
Thus, the edges~$u_1u_n < u_2u_n <\dots < u_{n-1}u_n$ are embedded onto the edges of the form~$v_ix$ for~$i\in[n-1]$. 
In fact,  since the \sco{} is larger decreasing, we must have that
$$u_i \mapsto v_{n-i} \qquad \text{for every } i\in [n-1]\,.$$
However, this yields a contradiction; indeed, while we have that~$u_1u_2 < u_1u_3$ in~$D_n$ we have~$v_{n-1}v_{n-2}> v_{n-1}v_{n-3}$ in the \sceo{} $K_n$.
\end{proof}

We use Proposition~\ref{prop::Dn} to prove that there is no tileable edge-ordering of~$K_4^-$. 

\begin{proposition} \label{prop::K4-}
    No edge-ordering of $K_4^-$ is tileable.
\end{proposition}

\begin{proof}
    To prove the proposition we will show that the only Tur\'anable edge-ordering of $K_4^-$ is in fact $D_4$, which, due to Proposition~\ref{prop::Dn} is not tileable. 

    As stated in \cite[Section 5]{gmnptv}, the only Tur\'anable edge-ordering of $C_4$ with vertices~$\{w_1,w_2,w_3,w_4\}$ is given by $w_1w_2<w_2w_3 < w_1w_4 < w_3w_4$; we denote this edge-ordered graph by~$C_4^{1243}$.
    Thus, in any Tur\'anable edge-ordering of~$K_4^-$ the underlying~$C_4$ must be a copy of~$C_4^{1243}$.
    Starting with such a copy of $C_4^{1243}$ we obtain a~$K_4^-$ by either adding the edge~$w_1w_3$ or~$w_2w_4$.
    
    Take an embedding of~$C_4^{1243}$ into the inverse min canonical ordering of~$K_4$ given by 
    $$  w_1\mapsto v_{i_1}\,, \quad
        w_2\mapsto v_{i_2}\,, \quad
        w_3\mapsto v_{i_3}\,, \quad \text{and} \quad
        w_4\mapsto v_{i_4}\,.
    $$
    We first show that this embedding is unique and given by~\eqref{eq:C4embedding} below.
    Suppose  that the edge~$w_1w_2$ is not embedded onto an edge containing~$v_1\in V(K_n)$; in other words,~$i_1\neq 1$ and~$i_2\neq 1$.
    Thus, there is a $j\in \{2,3,4\}$ such that $v_1 v_{j} = v_{i_3}v_{i_4}$. 
    This is a contradiction, since~$v_{i_1}v_{i_2} > v_1 v_{j}$ in the inverse min canonical ordering, while~$w_1w_2< w_3w_4$ in~$C_4^{1243}$.
    Hence, we have that either~$i_1=1$ or~$i_2=1$. 
    In the former case, since~$w_2w_3< w_1w_4$ then we have $v_{i_2}v_{i_3}< v_{i_1}v_{i_4} = v_1v_{i_4}$. 
    But this is a contradiction, because in the inverse min canonical ordering all edges containing $v_1$ are smaller than the edges not containing it. 
    Therefore, we must have that $i_2 = 1$.
    Further, observe that~$w_1w_2<w_2w_3$ means that~$v_1v_{i_1} < v_1v_{i_3}$, which in the inverse min ordering means that~
    \begin{align}\label{eq:3<1}
        i_3 < i_1\,.
    \end{align}
    Since~$i_1\leq 4$ and~$i_2=1$, we have~$2\leq i_3\leq 3$.
    Finally, observe that if~$i_3=2$, then we have~$v_{i_1}v_{i_4}=v_3v_4$. 
    But this is again a contradiction, since~$v_3v_4$ is the largest edge in the inverse min ordering of~$K_4$ while $w_1w_4 < w_3w_4$.
    Thus we get $i_3=3$, which together with~\eqref{eq:3<1}, implies that~$i_1=4$. 
    Summarizing, we have~$i_2=1$, $i_3=3$ and~$i_1=4$, which finally gives the embedding
    \begin{align}\label{eq:C4embedding}
        w_1\mapsto v_4\,, \quad
        w_2\mapsto v_1\,, \quad
        w_3\mapsto v_3\,, \quad \text{and} \quad
        w_4\mapsto v_2\,\,.
    \end{align}

    Thus, any Tur\'anable edge-ordering of~$K_4^-$ obtained by adding one edge to~$C_4^{1243}$ must be embedded into the inverse min canonical ordering of~$K_4$ via \eqref{eq:C4embedding}. 
    In this way, after adding the edge $w_2w_4$ or $w_1w_3$ to $C_4^{1243}$, the embedding~\eqref{eq:C4embedding} gives rise to the following edge-orderings of~$K_4^-$:
    \begin{align}
        w_1w_2&<w_2w_3 < w_2w_4 < w_1w_4 < w_3w_4 \qqand \label{eq:orderingD_4}\\
        w_1w_2&<w_2w_3 < w_1w_4 < w_3w_4 < w_1w_3\,, \label{eq:impossible}
    \end{align}
    respectively.
    The ordering \eqref{eq:orderingD_4} corresponds with the edge-ordering of~$D_4$, by taking~$u_1=w_2$, $u_2=w_1$, $u_3=w_3$, and $u_4=w_4$ (see the definition of~$D_4$ before Proposition~\ref{prop::Dn}). 
    
    For~\eqref{eq:impossible}, we shall prove that such an edge-ordering of~$K_4^-$ cannot be embedded into  the inverse max canonical ordering of~$K_4$, and therefore, it is not Tur\'anable.
    More precisely, we show that~$C_4^{1243}$ has only one possible embedding into the inverse max ordering of~$K_4$, but the embedding of the edge~$w_1w_3$ will lie in a different `position' than the one given by~\eqref{eq:impossible}.

    Let~$w_1'$, $w_2'$, $w_3'$, $w_4'$ be the vertices of~$\back{C}_4^{1243}$, the reverse ordering of $C_4^{1243}$, with edges 
    $$w_1'w_2'>w_2'w_3' > w_1'w_4' > w_3'w_4'\,.$$ 
    Here we now denote $\back{C}_4^{1243}$ by~$C_4^{4312}$. 
    Recall that the inverse max ordering of~$K_4$ with vertices~$\{v_1', v_2', v_3',v_4'\}$ corresponds with the reverse of the inverse min ordering on~$\{v_1,v_2,v_3,v_4\}$ by relabeling the vertices as~$v_1'=v_4$, $v_2'=v_3$, $v_3'=v_2$, and~$v_4'=v_1$. 
    Applying the symmetric reasoning as the one above, we have that there is only one possible embedding of~$C_4^{4312}$ into the inverse max ordering of~$K_4$.
    Namely, 
    \begin{align}\label{eq:C4maxembedding}
        w_1'\mapsto v_1'\,, \quad 
        w_2'\mapsto v_4'\,, \quad
        w_3'\mapsto v_2'\,, \quad \text{and} \quad
        w_4'\mapsto v_3'\, .
    \end{align}

    Moreover, notice that~$C_4^{1243}$ is isomorphic to~$C_4^{4312}$ by taking~$w_1=w_3'$, $w_2=w_4'$, $w_3=w_1'$, and~$w_4=w_2'$, where~$w_1, w_2, w_3, w_4$ are the vertices of~$C_4^{1243}$ as in the beginning of the proof.
    Thus, an embedding of $C_4^{1243}$ into an inverse max ordering of~$K_4$ must follow \eqref{eq:C4maxembedding} via this isomorphism to~$C_4^{4312}$.
    This corresponds to 
    \begin{align*}
        w_1\mapsto v_2'\,, \quad 
        w_2\mapsto v_3'\,, \quad
        w_3\mapsto v_1'\, \quad \text{and} \quad
        w_4\mapsto v_4'\,. \quad 
    \end{align*}
    Finally, the edge~$w_1w_3$ is embedded in this way onto~$v_1'v_2'$, which is the smallest edge of the inverse max ordering. 
    In other words we obtain,
    $$w_1w_3<w_1w_2<w_2w_3 < w_1w_4 < w_3w_4\,,$$
    which is incompatible with \eqref{eq:impossible}. 
\end{proof}

  
  
  
  

The following two propositions are useful to generate a tileable (or Tur\'anable) graph by appropriately adding a vertex and an edge to a tileable (or Tur\'anable) graph.

\begin{proposition}\label{lem::add_one_edge_turan}
Let $F$ be a  Tur\'anable edge-ordered graph and $v \in V(F)$
 a vertex  incident to the smallest edge in $F$.
 Let $F'$ be the edge-ordered graph obtained from $F$ 
by adding a new vertex $v'$ and an edge between $v$ and $v'$ smaller than all edges in $F$. Then $F'$ is Tur\'anable.
\end{proposition}

\begin{proof}
    Let $|F|:=f$ and $vu$ be the smallest edge in $F$. 
    We want to embed $F'$ into each canonical ordering of $K_{f+1}$.
    
    Observe that for the min ordering, inverse min ordering, and max ordering of $K_{f+1}$ we have 
    \begin{align}\label{eq:1i<ij}
        v_1v_i < v_iv_j \text{ for every distinct }i, j\in \{2,\dots, f+1\} .
    \end{align}
    For these canonical orderings we use that~$F$ is Tur\'anable and Fact~\ref{fact:selfidentical} to embed $F$ into $K_{f+1}\big[\{v_2, \linebreak[1] \dots, v_{f+1}\}\big]$ and then we embed~$v'$ onto $v_1$.
    Let~$i,j\geq 2$ be such that~$v$ and~$u$ are embedded in this way onto the vertices~$v_i$ and~$v_j$ respectively. 
    Since~$vu$ is the minimal edge in~$F$, then~$v_iv_j$ is minimal in our embedding of~$F$ into $K_{f+1}\big[\{v_2, \dots, v_{f+1}\}\big]$. 
    Thus, since $v'\mapsto v_1$ and~$v_1v_i < v_iv_j$ by~\eqref{eq:1i<ij}, this embedding gives rise to a copy of~$F'$ in these canonical edge-orderings of~$K_{f+1}$. 

    For  the inverse max ordering, we proceed as follows. 
    Let~$t\in [f]$ be such that there is an embedding of $F$ into an inverse max ordering of $K_{f}$ where $v_{t}$ plays the role of~$v$.
    Since~$F$ is Tur\'anable and due to Fact~\ref{fact:selfidentical}, we can embed $F$ into $K_{f+1}[\{v_1, \dots, v_{t-1}, v_{t+1},\linebreak[1] \dots, v_{f+1}\}]$ with $v_{t+1}$ playing the role of~$v$.
    We extend this embedding by assigning~$v'$ to $v_t$.
    In this way $v'v$ is mapped to $v_tv_{t+1}$ and~$uv$ is mapped to an edge of the form~$v_{t+1}v_i$ for $i\neq t$. 
    By the definition of the inverse max ordering we have~$v_tv_{t+1}<v_{t+1}v_i$, i.e., the embedding of the edge $vv'$ is smaller than the embedding of the edge $uv$. 
  Thus, the inverse max ordering of~$K_{f+1}$ contains a copy of~$F'$.
\end{proof}

\begin{proposition} \label{lem::add_one_edge}
    Let $F$ be a tileable edge-ordered graph and  $v\in V(F)$ a vertex incident to the smallest edge in $F$. 
     Let $F'$ be the edge-ordered graph obtained from $F$ 
by adding a new vertex $v'$ and an edge between $v$ and $v'$ smaller than all edges in $F$.
    Then $F'$ is  tileable.
\end{proposition}

\begin{proof}
    Let $|F|:=f$ and $uv$ be the smallest edge in $F$. 
    We want to embed $F'$ into each \sco{} of $K_{f+1}$. 
    We divide the proof into cases depending on the type of the \sco{}.
    
    For  smaller orderings of~$K_{f+1}$, 
    we use that~$F$ is Tur\'anable to first embed $F$ into a canonical ordering of the same type as the canonical part~$K_{f+1}-x$.  
    We then extend this embedding by setting $v' \mapsto x$. 
    The edge~$vv'$ is embedded onto an edge of the form~$xv_j$ with $j\in [f]$. 
    Thus, by definition of the smaller orderings, our embedding corresponds to a copy of~$F'$ in $K_{f+1}$. 
    
    In fact, in the argument above we only used that the smaller orderings satisfy
    \begin{align}\label{eq:smaller}
        xv_i < v_iv_j \text{ for every distinct }i,j\in [f]\,,
    \end{align}
    since we only need that the embedding of~$vv'$ is smaller than the embedding of the smallest edge in~$F$. 
    More precisely, observe that if~$v'\mapsto x$ and~$v\mapsto v_i$ for some~$i\in [f]$, then the edge~$vv'$ in~$F'$ is sent to the edge $xv_i$ and the minimal edge of $F$,~$uv$, is sent to an edge of the form~$v_iv_j$ in $K_{f+1}$ for a~$j\in [f]\setminus \{i\}$.
    Thus, if \eqref{eq:smaller} holds, then the embedding of~$vv'$ is smaller than the embedding of the smallest edge in~$F$, yielding a copy of~$F'$.
    It is easy to check that~\eqref{eq:smaller} holds for a middle increasing ordering whose canonical part is an inverse max ordering.
    Indeed, following the labelings in Definitions~\ref{def:canonical} and~\ref{def:starcanonical}, for a middle increasing ordering whose canonical part is an inverse max ordering we have 
    \begin{align*}
        L_4(xv_i)&=2fi < 2fj-i+f = L_4(v_iv_j) \, \quad \text{ for } 1\leq i<j \leq f\,, \qand\\
        L_4(xv_i)&=2fi < 2fi-j+f = L_4(v_iv_j) \, \quad \text{ for } 1\leq j < i\leq f\,.
    \end{align*}
    Thus, for this~\sco{} we can proceed as described above.

    We shall now address the remaining \sco{}s of~$K_{f+1}$, these are: all larger orderings and all middle increasing orderings except when the canonical part is an inverse max ordering.
    For these \sco s of~$K_{f+1}$, we will proceed differently depending on how~$F$ embeds into a~\sceo{} $K_f$ of the same type as~$K_{f+1}$.
    Recall that such embeddings exist due to Fact~\ref{fact:selfidentical_star} and because $F$ is a tileable edge-ordered graph. 
    
    First, note that for all remaining \sco s 
    \begin{align}\label{eq:embedding}
        v_1v_i < xv_i \quad \text{for every }2\le i \le f\,.
    \end{align}
    Indeed, for the larger orderings this follows directly from the definition. 
    For the middle increasing orderings whose canonical part is not the inverse max ordering, we just need to check the following inequalities given by the labelings in Definitions~\ref{def:canonical} and~\ref{def:starcanonical} for $2\leq i\leq f$:
    \begin{itemize}[leftmargin=0.9cm]
        \item for the canonical part being a min ordering $L_1(v_1v_i) = 2f + i -1 < 2fi = L_1(v_ix) $\,,
        \item for the canonical part being a max ordering $L_2(v_1v_i) = (2f-1)i + 1< 2fi= L_2(v_ix)$\,, and
        \item for the canonical part being an inverse min ordering $L_3(v_1v_i) = (2f+1)-i < 2fi = L_3(v_ix)$\,.
    \end{itemize}
    Note that \eqref{eq:embedding} does not hold for smaller orderings or for middle increasing orderings whose canonical part is an inverse max ordering. 
    
    Now suppose that in an embedding of~$F$ into a \sceo{} $K_f$ of the same type as~$K_{f+1}$, vertex $u$ is embedded as the special vertex $x$.
    Then, we embed~$F'$ into the~\sceo{} $K_{f+1}$ by first embedding $F$ into $K_{f+1}[\{x,v_2, \dots, v_{f}\}]$ with~$u$ as the special vertex, and then mapping $v'$ to $v_1$.
    To check that this is  an embedding of~$F'$ into $K_{f+1}$ observe that $v'v$ is embedded onto $v_1v_i$ and~$uv$ is embedded onto~$xv_i$, for some~$i\geq 2$. 
    Since $uv$ is the smallest edge in $F$, the edge~$xv_i$ is the smallest in our embedding of $F$ into~$K_{f+1}[\{x,v_2, \dots, v_{f}\}]$.
    Due to~\eqref{eq:embedding}, $v_1v_i < xv_i$, and so the edge $v'v$ is 
    mapped to an edge smaller than all the edges in our copy of $F$.
    This yields a copy of $F'$ in~$K_{f+1}$.

    Next suppose that in an embedding of~$F$ into a \sceo{} $K_f$ of the same type as~$K_{f+1}$, $v$ is embedded onto the special vertex $x$. 
    If the \sco{} is increasing, then we embed $F$ into $K_{f+1}[\{x,v_2, \dots, v_{f}\}]$ with~$v$ as the special vertex, and map $v'$ to $v_1$.
    Thus, the edge~$vv'$ is embedded onto~$xv_1$ and~$uv$ is embedded onto an edge $xv_i$ for some $i\geq 2$. 
    Since the \sco{} is increasing we have~$xv_1<xv_i$.
    As before this yields an embedding of~$F'$ into~$K_{f+1}$.
    If the \sco{} is decreasing we proceed analogously by first embedding $F$ into $K_{f+1}[\{x,v_1, \dots, v_{f-1}\}]$ and then extending that embedding by assigning $v'$ to $v_f$.
    
    Finally, suppose that in all embeddings of~$F$ into a \sceo{} $K_f$ of the same type as~$K_{f+1}$, neither $u$ nor $v$ is embedded as the special vertex $x$.
    Then we proceed similarly to the proof of Proposition~\ref{lem::add_one_edge_turan} above. 
    If the canonical part is a min ordering, an inverse min ordering, or a max ordering, then we first embed~$F$ into~$K_{f+1}\big[\{x, v_2, \dots, v_{f+1}\}\big]$ and then~$v'$ onto~$v_1$. 
    Let~$i,j\geq 2$ be such that~$v$ and $u$ are embedded in this way onto vertices~$v_i$ and~$v_j$ respectively, both in the canonical part of $K_{f+1}\big[\{x, v_2, \dots, v_{f+1}\}\big]$.
    Then, since \eqref{eq:1i<ij} holds in this context for the edge-ordering of the canonical part, we have that~$v_1v_i<v_iv_j$. 
    Hence,  $v'v$ is mapped to an edge, $v_1v_i$, that is smaller than the 
    edge $v_iv_j$ that $uv$ is mapped to.
    As before this yields an embedding of~$F'$ into~$K_{f+1}$.
    If the canonical part is an inverse max ordering, let~$t\in[f]$ be such that there is an embedding of~$F$ into the \sceo{} $K_{f}$ of the same type as~$K_{f+1}$ for which~$v\mapsto v_t$.
    Then we embed $F$ into $K_{f+1}[\{x,v_1, \dots, v_{t-1}, v_{t+1}, \dots, v_{f}\}]$ in such a way that~$v\mapsto v_{t+1}$.
    We are assuming that in every embedding of~$F$ into a~\sceo{} $K_f$ of the same type as~$K_{f+1}$, neither $v$ nor $u$ is embedded as the special vertex~$x$; so there is an $i\in [f]\setminus\{t\}$ such that~$u\mapsto v_i$ in our embedding. 
    Extend this embedding by assigning~$v'$ to $v_t$.
    In this way we have 
    $$v'v\mapsto v_tv_{t+1} \qqand uv \mapsto v_{t+1}v_i\,.$$
    In the inverse max ordering we have~$v_tv_{t+1}<v_{t+1}v_i$, which means that the  edge $v'v$ is mapped to is smaller than edge $uv$ is mapped to. 
    As before this yields a copy of~$F'$ into~$K_{f+1}$.
\end{proof}

Using Theorems~\ref{thm:turanable} and~\ref{thm:character} it is easy to see that any Tur\'anable edge-ordered graph becomes tileable after adding an isolated vertex.
More interestingly, the next proposition implies that given any connected Tur\'anable graph~$F$ we can obtain a connected tileable graph on~$\vert F\vert +2$ vertices that contains $F$.

Given a Tur\'anable edge-ordered graph $F$ on~$f$ vertices, we say a vertex~$v\in V(F)$ is \emph{minimal} if it plays the role of $v_1$ in an embedding of~$F$ into a min ordering of~$K_f$.
Similarly, we say that~$v$ is \emph{maximal}  if it plays the role of $v_f$ in an embedding of~$F$ into a max ordering of~$K_f$. 
By Theorem~\ref{thm:turanable} a Tur\'anable graph always contains at least one minimal and one maximal vertex\footnote{We highlight that there might be more than one minimal (resp. maximal) vertex, as there might be more than one embedding of $F$ into a min (resp. max) ordering. 
For example, in a monotone path $u_1u_2u_3u_4$, we have that $u_1$ and~$u_2$ can play the role of~$v_1$ in a min ordering.}.
Observe that the edges incident to a minimal (resp. maximal) vertex are always smaller (resp. larger) than the edges not incident to it.

We show that starting with a Tur\'anable graph we can add two pendant edges, one to a minimal vertex and one to a maximal vertex, and obtain a tileable graph. 
This result, together with the example of a Tur\'anable  graph $D_n$ that is not tileable (see Proposition~\ref{prop::Dn}), implies the perhaps surprising property that being tileable is not closed under taking connected subgraphs.

\begin{proposition}\label{prop::add_two_edges}
Let $F$ be an edge-ordered Tur\'anable graph with $\vmin, \vmax\in V(F)$ being distinct non-isolated
minimal and maximal vertices respectively. 
Let $F'$ be constructed by adding two new vertices $\umin, \umax$ and the edges $\umin\vmin$ and~$\umax\vmax$ such that $\umin\vmin$ is smaller than all other edges and $\umax\vmax$ is larger than all other edges.
Then $F'$ is tileable.
\end{proposition}
\begin{proof}
    Let $f:=|F|$. 
    As $F$ is Tur\'anable, by Proposition~\ref{lem::add_one_edge_turan} we have that $F' - \umax$ is Tur\'anable as well.
    Applying Proposition~\ref{lem::add_one_edge_turan} to the reverse of~$F'-\umin$ we get that~$F'-\umin$ is Tur\'anable too. 
    Thus, due to Theorem~\ref{thm:turanable} we can embed $F' - \umin$ and $F' - \umax$ into any canonical ordering of $K_{f+1}$.
    We will use these embeddings to find embeddings of $F'$ into each \sco{} of $K_{f+2}$.

   
    For the smaller orderings of~$K_{f+2}$, we first embed $F'-\umin$ into the canonical part $K_{f+2} - x$, and then embed $\umin$ as the special vertex $x$.
    In this way, the edge $\umin\vmin$ is embedded onto an edge of the form $xv_i$; therefore, by definition of the smaller orderings, the edge $\umin\vmin$ is embedded onto is smaller than all edges in the embedding of $F'-\umin$.
    This gives rise to a copy of~$F'$.
    
    For the larger orderings of~$K_{f+2}$ the proof is analogous, by embedding $F'-\umax$ into the canonical part $K_{f+2} - x$ and then embedding $\umax $ onto $x$.

   For the middle increasing \sco s of~$K_{f+2}$, we now split into subcases depending on its canonical part.

    If the canonical part is a min ordering, since $\vmin$ is a minimal vertex in $F$, there is an embedding of $F$ into~$K_{f+2}[\{v_1,\dots,v_f\}]$ such that $\vmin \mapsto v_1$.
    Let~$i\in[f]\setminus \{1\}$ be such that $\vmax \mapsto v_i$ in that embedding.
    Observe that, for every edge~$w_1w_2$ in~$F$ such that~$w_1\mapsto v_j$ and $w_2\mapsto v_k$ for a pair of indices~$j,k\in [f]\setminus \{i\}$, we have 
    \begin{align}\label{eq:maximaledge}
        v_jv_k < v_iv_{f+1}
    \end{align}
    in the edge-ordering of~$K_{f+2}$.
    To see this, observe that since~$\vmax$ is maximal and not isolated in $F$, $\vmax$ must be contained in the maximal edge of~$F$, and hence, the embedding of the maximal edge must be of the form~$v_iv_\ell$ for some~$\ell\in [f]\setminus\{i\}$. 
    Thus, if~\eqref{eq:maximaledge} does not hold for some edge~$w_1w_2$ in~$F$, then 
    $$v_jv_k> v_iv_{f+1} > v_iv_\ell\,,$$ 
    where the last inequality holds since the canonical part is a min ordering and~$\ell<f+1$.   
    However, this is a contradiction since the maximal edge in $F$
    is embedded onto $v_iv_\ell$.
    Now we extend this embedding to an embedding of~$F'-\umin$ by taking~$\umax\mapsto v_{f+1}$. 
    Indeed, the edge~$\umax\vmax$ is embedded onto~$v_iv_{f+1}$ which, due to~\eqref{eq:maximaledge}, is larger than  any edge in our copy of~$F$, implying a copy of~$F'-\umin$ in~$K_{f+1}$.
    Finally, extend the embedding further by taking $\umin\mapsto x$.
    Observe that the edge~$\umin\vmin$ is embedded in this way onto the edge~$xv_1$. 
    Moreover, by Remark~\ref{rem:middle}, $xv_1<v_1v_2$ and $v_1v_2$ is the smallest edge in the canonical part by Definition~\ref{def:canonical}.
    Therefore, the edge $\umin\vmin$ is embedded onto is smaller than  all other edges used.
    Thus, we find a copy of~$F'$.

    If the canonical part is an inverse min ordering, we embed~$F'-\umax$ into the canonical part and then take~$\umax \mapsto x$.
    Let~$v_i$ be the vertex $\vmax$ is embedded onto (where~$i\in [f]$). 
    Note that 
    \begin{align}\label{eq:invmin}
      xv_i > \max\{v_iv_j \colon j\in [f]\setminus \{i\}\}\,.  
    \end{align}
    Indeed, using the labelings given by Definitions~\ref{def:canonical} and \ref{def:starcanonical} we have $L_3(xv_i) = 2fi > (2f+1)i - j =L_3(v_iv_j)$ for every~$i<j\leq f$ and $L_3(xv_i) = 2fi > (2f+1)j - i =L_3(v_iv_j)$ for every~$1\leq j < i$.
    Thus, \eqref{eq:invmin} implies that 
    the edge $xv_i$ that $\umax\vmax$ is embedded onto is larger than
    any of the edges in our copy of $F'-\umax$
that contain~$\vmax$. 
    Since~$\vmax$ is a maximal non-isolated vertex in $F$, $\vmax$ is contained in the maximal edge of~$F$. 
    The maximal edge of~$F$ is also the maximal edge of~$F'-\umax$ and therefore, the edge $xv_i$ that $\umax\vmax$ is embedded onto is larger any of the edges in our copy of $F'-\umax$.
    As before, this yields a copy of~$F'$ in~$K_{f+2}$.     

    Finally, if the canonical part is a max ordering or an inverse max ordering we argue as before, but for the reverse graph~$\back{F'}$. 
    More precisely, note first that for~$\back{F}$ the vertices $\vmin$ and~$\vmax$ are maximal and minimal respectively. 
    Moreover, if~$F''$ is constructed from~$\back{F}$ by adding two new vertices $\underline w, \overline w$ and the edges $\underline w\vmin$ and~$\overline w\vmax$ such that $\underline w\vmin$ is larger than all other edges and~$\overline w\vmax$ is smaller than all other edges, then $F''$ is precisely the reverse of~$F'$. 
    By the argument above, a middle increasing ordering of~$K_{f+2}$ whose canonical part is a min or an inverse min ordering contains a copy of~$F''$. 
    Hence, the reverse of that ordering contains a copy of~$F'$. 
    We conclude by noticing that reverse of a  middle increasing ordering whose canonical part is the min (resp. inverse min) ordering is the middle increasing ordering whose canonical part is the max (resp. inverse max) ordering.
\end{proof}

Proposition~\ref{prop::K4-} implies that no edge-ordering of $K_4^-$ is tileable.
In contrast, the following corollary of Propositions~\ref{lem::add_one_edge} and~\ref{prop::add_two_edges}
asserts that there are connected tileable edge-ordered graphs containing $K_4 ^-$.
Recall $D_4$ is a Tur\'anable edge-ordering of $K_4^-$; further notice $D_4$ has unique minimal and maximal vertices, and they are distinct.

\begin{corollary}\label{cor:K_4-}
For every $n\geq 6$ there is a connected~$n$-vertex tileable edge-ordered graph~$F_n$ with~$K_4^-\subseteq F_n$. 
\end{corollary}

\begin{proof}
We first use induction to show that the result holds for every \emph{even} $n \geq 6$.
For~$n=6$, apply Proposition \ref{prop::add_two_edges} with~$F:=D_4$,  and let $F_6$ be the resulting edge-ordered graph. 
Since~$F_6$ is tileable and~$K_4^-\subseteq F_6$, we establish the base case. 
Notice that one of the new vertices in $F_6$ is minimal, the other  is a maximal vertex.
Similarly, suppose that~$F_n$ is a connected $n$-vertex tileable edge-ordered graph with distinct minimal and maximal vertices so that~$K_4^-\subseteq F_n$.
Then we apply Proposition~\ref{prop::add_two_edges} with~$F_n$ playing the role of~$F$ and let~$F_{n+2}$ be the  output of this proposition. 
Notice that $F_{n+2}$ is a connected $(n+2)$-vertex tileable edge-ordered graph with~$K_4^-\subseteq F_n\subseteq F_{n+2}$. Moreover, $F_{n+2}$ will contain distinct minimal and maximal vertices (the two new vertices).

Since the corollary holds for all even $n \geq 6$, we may apply Proposition~\ref{lem::add_one_edge}
to deduce the result for all odd $n \geq 6$.
\end{proof}

In Proposition~\ref{prop::add_two_edges} we obtain a tileable edge-ordered graph from a Tur\'anable edge-ordered graph by adding \emph{two} pendant edges.
The following proposition shows that adding only one such pendant edge is, in general, not enough to create a tileable edge-ordered graph.
Recall we write $u_1, \dots, u_n$ for the vertices of $D_n$ where 
$u_1$ and $u_n$ are the unique minimal and maximal vertices in $D_n$, respectively.

\begin{proposition}
For $n\ge 4$, let $D_n^+$ be the edge-ordered graph obtained from $D_n$ by adding a new vertex $w$ and the edge $u_nw$, larger than all the  edges in $D_n$.
 Let $D_n^-$ be the edge-ordered graph obtained from $D_n$ by adding a new vertex $u$ and the edge $u_1u$, smaller than all the  edges in $D_n$. Then 
neither $D_n^+$ nor
 $D_n^-$ are tileable. 
\end{proposition}

\begin{proof}
We only consider $D_n^+$ as the argument for  $D_n^-$ is analogous.
Suppose for a contradiction  there is an embedding of~$D_n^+$ into a smaller decreasing ordering of~$K_{n+1}$ whose canonical part is a min ordering.
First, since $u_1u_n<u_iu_n$ for $1<i<n$ and $u_nw$ is the largest edge in $D_n^+$,~$u_1$ is the only vertex in~$D_n^+$ such that all edges incident to it are smaller than all other edges. 
Note that this means we must have that $u_1\mapsto x$.
Recall that~$u_1u_2<\dots<u_1u_n$ in $D_n$ and that, since the \sco{} of~$K_{n+1}$ is smaller decreasing, $v_1x > \dots > v_nx$. 
Thus, given~$1<i<j\leq n$, 
\begin{align*}
\text{if $u_i \mapsto v_k$ and $u_{j}\mapsto v_{\ell}$ then~$\ell<k$.}
\end{align*}
In particular, if we take~$i,j,k\in [n]$ such that
$$u_n \mapsto v_i\,, \qquad u_3 \mapsto v_j\,, \qqand u_2 \mapsto v_k,$$
then~$i<j<k$. 
However, this is a contradiction, because while~$u_nu_2<u_nu_3$ in~$D_n^+$, we have~$v_i v_k > v_i v_j$ in~$K_{n+1}$. 
\end{proof}

In the following two propositions we study the tileability of {monotone cycles}.
Recall that we say that an edge-ordered cycle~$C_n$ with $V(C_n)=\{u_1,\dots, u_n\}$ is \emph{monotone} if the edges are ordered as $u_1u_2<u_2u_3<\dots<u_{n-1}u_n<u_nu_1$.

\begin{proposition}\label{prop::monotone_odd}
Monotone cycles of odd length are tileable. 
\end{proposition}

\begin{proof}
It suffices to find a spanning monotone cycle in every \sco{} of $K_{n+1}$ where $n$ is even. 
For this, we show that every canonical ordering of $K_n$ contains an embedding of the monotone spanning path which can be extended by adding the special vertex $x$ on both ends so that the resulting cycle is monotone.

We now define four paths in the canonical orderings with vertex set~$\{v_1,\dots, v_n\}$, and  state in which canonical orderings they are in fact monotone paths.
\begin{itemize}
    \item \emph{Ordinary}: $v_1v_2v_3\dots v_n$ is monotone in all four canonical orderings.
    \item \emph{Small}: $v_2v_3\dots v_nv_1$ is monotone in the inverse max ordering.
    \item  \emph{Big}: $v_nv_1v_2\dots v_{n-1}$ is monotone in the inverse min ordering.
    \item
    \emph{Jumpy}: $v_{n/2+1}v_1v_{n/2+2}v_2 \cdots v_nv_{n/2}$ is monotone in the min ordering and max ordering.
\end{itemize}

        
        
        
For each~\sco{} of~$K_{n+1}$, we now show how to extend one of the previous monotone paths into a spanning monotone cycle using the special vertex~$x$.

For all larger/smaller decreasing orderings, we simply extend the ordinary path by adding the special vertex $x$ `between' $v_n$ and $v_1$. The resulting cycle is monotone since, by Definition~\ref{def:starcanonical}, 
\begin{itemize}
    \item for larger decreasing orderings $ v_1v_2 < \ldots < v_{n-1}v_{n} < v_nx < xv_1 \,;$
    \item for smaller decreasing orderings $ v_nx< xv_1< v_1v_2 < \ldots < v_{n-1}v_n \,.$
\end{itemize}

The remaining \sco s are all increasing. We split the analysis into cases depending on their canonical part.

Suppose first that the canonical part is a min or a max ordering. For the middle increasing ordering observe  Remark~\ref{rem:middle} implies that for the min and max orderings,~$xv_1$ and~$xv_n$ are the smallest and largest edges respectively.
Then, we simply take the ordinary path and add the special vertex between $v_n$ and $v_1$ to get a monotone cycle 
$$ xv_1 < v_1v_2 < \ldots < v_{n-1}v_n < v_nx \,.$$
If the ordering is smaller or larger increasing we extend a jumpy path by adding the special vertex $x$ between $v_{n/2}$ and $v_{n/2+1}$ as we have~$xv_{n/2}<xv_{n/2+1}$ for all increasing orderings. Observe that the resulting cycle is monotone, since 
\begin{itemize}
    \item for larger increasing orderings 
    $ v_{n/2+1}v_1 < \ldots < v_{n}v_{n/2} < v_{n/2}x < xv_{n/2+1} \,;$
    \item for smaller increasing orderings 
    $ v_{n/2}x < xv_{n/2+1} < v_{n/2+1}v_1 < \ldots < v_{n}v_{n/2} \,.$
\end{itemize}

Suppose now that the canonical part is an inverse min ordering. We extend the big path by adding the special vertex between $v_{n-1}$ and $v_n$.
By Definition~\ref{def:starcanonical} and Remark~\ref{rem:middle}, observe that 
for the larger and middle increasing orderings, 
$$ v_nv_1 < \ldots < v_{n-2}v_{n-1} < v_{n-1}x < xv_{n}\,,$$ 
while for the smaller increasing ordering,
$$ v_{n-1}x < xv_{n} < v_nv_1 < \ldots < v_{n-2}v_{n-1} \,.$$ 

Finally, suppose the canonical part is an inverse max ordering; we extend the small path by
adding the special vertex between $v_1$ and $v_2$. Indeed,
by Definition~\ref{def:starcanonical} and Remark~\ref{rem:middle}, observe that 
for the smaller and middle increasing orderings, 
$$ v_1 x < xv_2 < v_2v_3 < \ldots <  v_{n} v_1 \,,$$ 
while for the larger increasing ordering,
$$ v_2v_3 < \ldots <  v_{n} v_1 <v_1 x<xv_2 \,. \eqno\qedhere$$

\end{proof}

In stark contrast to Proposition~\ref{prop::monotone_odd}, the next result states that 
 monotone cycles of even length are not  Tur\'anable, let alone tileable.

\begin{proposition}\label{lem::monotone_even}
Monotone cycles of even length are not Tur\'anable. 
\end{proposition}

\begin{proof}
    By Theorem~\ref{thm:turanable}, it suffices to  show that there is no spanning monotone cycle in the min canonical ordering of $K_n$ for $n$ even. 
    We will proceed by induction on $n$. 
   
    Before this, we first show that in the min ordering of $K_n$,
    \begin{align}\label{eq:evencycles}
    \text{if $v_iv_j<v_jv_k$ then~$i<k$.}    
    \end{align}
    Indeed, suppose~$k<i$. Using the standard labeling of Definition~\ref{def:canonical}, we have that if~$j<k$ then~$2nj+i-1=L_1(v_iv_j) < L_1(v_jv_k) = 2nj+k-1$, which is a contradiction. 
    If $k<j<i$, then 
 $   2nj+i-1=L_1(v_iv_j)
        <L_1(v_jv_k)
        =2nk+j-1 ,$
    which implies that $2n(j-k) < j-i$; this
     is a contradiction, since $k<j$ while~$j<i$.     
    Finally, if $k<i<j$, then $2ni+j-1=L_1(v_iv_j)<L_1(v_jv_k)=2nk+j-1$, which again is a contradiction.

     Let~$C^{\text{mon}}_4$ be a monotone cycle of length four with vertices $u_1,u_2,u_3,u_4$ and edges ordered as~$u_1u_2<u_2u_3<u_3u_4<u_4u_1$.
Suppose there is an embedding of~$C_4^{\text{mon}}$ into the min ordering of~$K_4$ and let~$i,k\in [4]$ be such that 
    $$u_1\mapsto v_i \qqand u_3\mapsto v_k\,.$$
    Since~$u_1u_2<u_2u_3$ and due to \eqref{eq:evencycles}, we have~$i<k$, but similarly, since $u_3u_4<u_4u_1$, we have~$k<i$, a contradiction.
    
    Now suppose that the min ordering of $K_n$ does not contain a spanning monotone cycle $C_n^{\text{mon}}$ for some even $n\ge 4$.
    Let~$\{u_1, \dots, u_{n+2}\}$ be the vertex set of a monotone cycle~$C_{n+2}^{\text{mon}}$, with edges ordered as~$u_1u_2<\dots<u_{n+1}u_{n+2}<u_{n+2}u_1$. Suppose for contradiction  there is an embedding~
    $$\varphi\colon V(C_{n+2}^{\text{mon}}) \longrightarrow V(K_{n+2})$$
    of $C_{n+2}^{\text{mon}}$
    into the min ordering of~$K_{n+2}$.
    First, we shall check that for any three vertices~$v_i$, $v_j$, and~$v_k$ in the min ordering of~$K_{n+2}$, 
    \begin{align}\label{eq:evencycle2}
        \text{if $v_iv_j < v_jv_k$ then for every~$v_\ell\in V(K_{n+2})\setminus\{v_i,v_k\}$ we have $v_iv_\ell < v_kv_\ell$}\,.
    \end{align}
    Indeed, since~$v_iv_j < v_jv_k $, we have that \eqref{eq:evencycles} yields~$i<k$.
    If we suppose~$v_kv_\ell<v_iv_\ell$, then again \eqref{eq:evencycles} implies that~$k<i$, which is a contradiction, and therefore~\eqref{eq:evencycle2} follows.
        
    Due to \eqref{eq:evencycle2}, and since~$\phi(u_1)\phi(u_2)<\phi(u_2)\phi(u_3)$ in~$K_{n+2}$, we have~$\phi(u_1)\phi(u_4) < \phi(u_3)\phi(u_4)< \phi(u_4)\phi(u_5)$.
    Hence, we have that
    $$\phi(u_1)\phi(u_4) < \phi(u_4)\phi(u_5) < \phi(u_5)\phi(u_6) <\dots <\phi(u_{n+1})\phi(u_{n+2}) < \phi(u_{n+2})\phi(u_{1})\,,$$
    which is a copy of a monotone cycle of length~$n$ embedded into the edge ordered graph induced by the vertices~$V(K_{n+2})\setminus \{\phi(u_2),\phi(u_3)\}$. 
    But this is a contradiction to our induction hypothesis since, due to Fact \ref{fact:selfidentical},~$V(K_{n+2})\setminus \{\phi(u_2),\phi(u_3)\}$ induces a min ordering of~$K_n$. 
\end{proof}

\subsection{Proof of  Theorem~\ref{thm:character}}\label{subsec:proof}
First we prove the following lemma that provides an alternative characterization of tileable edge-ordered graphs. 

\begin{lemma}\label{lemma:character}
An edge-ordered graph~$F$ is tileable if and only if there exists an~$n\in \mathbb N$ such that the following holds.
Every edge-ordering of~$K_n$ such that $K_n-x$ is canonical for some vertex~$x\in V(K_n)$ contains a copy of~$F$ that covers $x$.
\end{lemma}

\begin{proof}
For the `forwards direction', suppose that there is no~$n\in \mathbb N$ satisfying the property described in the lemma.
That is, for every~$n\in\mathbb N$ there is an edge-ordering of the complete graph~$K_{n}$ such that~$K_n-x$ is canonically edge-ordered for some vertex~$x\in V(K_n)$ and~$x$ is not contained in any copy of~$F$.
In particular, none of these edge-ordered complete graphs contain an~$F$-tiling covering~$x$, and so by definition
 $F$ is not tileable. 

For the `backwards direction', let~$n\in \mathbb N$ be as in the statement of the lemma and set $f:=\vert V(F)\vert$.
We shall prove that $F$ is tileable, that is, there exists a $t\in \mathbb N$ such that every edge-ordering of~$K_t$ contains a perfect~$F$-tiling.
Note first that the property of~$n$ guarantees that every canonical edge-ordering of~$K_{n}$ contains a copy of~$F$. 
In particular, Fact~\ref{fact:selfidentical} implies that for every~$\ell\in \mathbb N$, every canonical edge-ordering of~$K_{\ell f}$ contains a perfect~$F$-tiling.
Further, given $k\geq n$ where $k$ is divisible by~$f$, if~$K_{k}$ is such that $K_{k}-x$ is canonically edge-ordered for some vertex~$x\in V(K_{k})$, then $K_k$ contains a perfect~$F$-tiling. 
Indeed, by the property of~$n$,~$K_{k}$ contains a copy $F'$ of $F$ with~$x\in V(F')$; hence, as $K_{k}\setminus V(F')$ is canonically edge-ordered, the discussion above implies that $K_{k}\setminus V(F')$, and thus $K_{k}$, contains a perfect $F$-tiling.

Pick~$k\geq n$ such that~$k$ is divisible by~$f$ and let~$m\in \mathbb N$ be the output of Proposition~\ref{prop:canonical} on input $k-1$. 
Fix~$t:=(m-1)k$ and let~$K:=K_t$ be arbitrarily edge-ordered.
Apply Proposition~\ref{prop:canonical} iteratively~$m-1$ times to find vertex-disjoint copies of $K_{k-1}$ in $K$, each of them canonically edge-ordered.
Let~$K_{k-1}^{(1)}, \dots, K_{k-1}^{(m-1)} \subseteq K$ be these copies and observe that exactly~$m-1$ vertices remain uncovered in $K$. 
That is, there are vertices~$x_1,\dots, x_{m-1}$ such that~$V(K)=\bigcup_{i\in [m-1]} V(K_{k-1}^{(i)})\cup \{x_i\}$. 
By the discussion above, for every~$i\in [m-1]$, $K[V(K_{k-1}^{(i)})\cup \{x_i\}]$ contains a perfect $F$-tiling and hence, $K$ contains a perfect~$F$-tiling as well, as required.
\end{proof}

In the proof of Theorem~\ref{thm:character} we deal with canonical orderings of~$K_n$ with vertex set~$\{v_1, \dots, v_n\}$. 
Let~$U\subseteq V(K_n)$ be a subset  of size~$k\leq n$ such that~$U=\{v_{i_1}, \dots, v_{i_k}\}$ where~$j < k$ implies~$i_j<i_k$. 
Whenever we say that we~\emph{relabel the vertices of $U$}, we mean that we will denote $v_{i_j}$ simply as~$v_j$ (and we will restrict our attention to this subset of the original vertex set).

\begin{proof}[Proof of Theorem~\ref{thm:character}]
Suppose $F$ is tileable; by definition there is some $n \in \mathbb N$ so that 
in any \sco{} of $K_{n+1}$ there is a perfect $F$-tiling.
In such a perfect $F$-tiling there is a copy $F'$ of $F$  that contains the special vertex $x$.  
Fact~\ref{fact:selfidentical_star} implies that $K_{n+1}[V(F')]$ is \sceo{} with the same type as~$K_{n+1}$. 
Thus, this implies every \sco{} of $K_f$ contains a copy of~$F$.

For the other direction, suppose every \sco{} of~$K_f$ contains a copy of~$F$. Our aim is to show that $F$ is tileable. 
By Lemma~\ref{lemma:character}, it suffices to prove that there is an~$n\in \mathbb N$ such that every edge-ordering of~$K_{n+1}$ for which~$K_{n+1}-x$ is canonically ordered for some vertex~$x\in V(K_{n+1})$, contains a copy of~$F$ that covers $x$. 

The cases~$f=2,3$ are trivial, so we may assume~$f\geq 4$.
Let~$n\in \mathbb N$ be sufficiently large compared to~$f\geq 4$ and where $\sqrt{n-1} \in \mathbb N$.
Let~$\{x, v_1,\dots, v_{n}\}$ be the vertices of an edge-ordered complete graph~$K_{n+1}$, such that~$K_{n+1} - x$ is canonically ordered. 
Our goal is to find a subgraph~$K_f\subseteq K_{n+1}$ containing $x$ such that~$K_f$ is \sceo{}. Indeed, by our assumption this $K_f$ contains a copy of $F$, and so $K_{n+1}$ contains a copy of $F$ that covers $x$, as desired.

\smallskip

Observe that an application of the Erd\H os--Szekeres 
Theorem~\cite{ErdosSzekeres} to the sequence of edges~$\{xv_i\}_{i\in [n]}$ yields a monotone subsequence. 
More precisely, there is a set $I\subseteq [n]$ of size at least~$\sqrt{n-1}+1$ such that the sequence~$\{xv_i\}_{i\in I}$ is monotone.
Further, let~$V_I:=\{v_i\}_{i\in I}$ and consider the  $3$-coloring~$c:E(K_{n+1}[V_I]) \to \{B,M,S\}$ of the edges of $K_{n+1}[V_I]$ defined as follows: 
for~$i, j \in I$ with~$i<j$, let
\begin{align*}
    c(v_iv_j) := 
    \begin{cases}
        B \qquad &\text{if }xv_i, xv_j > v_iv_j\,, \\
        M \qquad &\text{if }xv_i < v_iv_j < xv_j  \text{ or } xv_j < v_iv_j < xv_i\,, \text{and} \\
        S \qquad &\text{if } xv_i, xv_j < v_iv_j\,.
    \end{cases}
\end{align*}
As $n$ is sufficiently large, Ramsey's Theorem implies that there is a monochromatic clique~$\widetilde K$ on  $\ell :=f^2 -4f +5$ vertices.
Relabeling the vertices of $V(\widetilde K)$ we take~$V(\widetilde K)= \{v_1,\dots,v_\ell\}$ and thus we have
\begin{enumerate}
    \item $\widetilde K$ is canonically ordered;
    \item $\{xv_i\}_{i\in [\ell]}$ is a monotone sequence;
    \item exactly one of the following holds: 
    \begin{enumerate}
        \item $xv_i, xv_j > v_iv_j$ for every $1\leq i<j\leq \ell$, \label{alt:large}
        \item $xv_i, xv_j < v_iv_j$ for every $1\leq i<j\leq \ell$, or \label{alt:small}
        \item $xv_i < v_iv_j < xv_j$ or $xv_j < v_iv_j < xv_i$ for every $1\leq i<j \leq \ell$. \label{alt:middle}
    \end{enumerate}
\end{enumerate}
We shall prove that~$K_{n+1}[V(\widetilde K)\cup\{x\}]$ contains a \sceo{} copy of~$K_f$ containing the vertex~$x$, as desired. 
We split the rest of the proof into cases depending on whether the sequence $\{xv_i\}_{i\in [\ell]}$ is increasing or decreasing, and depending on which of \eqref{alt:large}, \eqref{alt:small}, and \eqref{alt:middle} holds. 

\begin{enumerate}[wide, labelwidth=!, labelindent=0pt, label=\caselabel]
    \medskip
    \item  \label {case:declarge} The sequence $\{xv_i\}_{i\in[\ell]}$ is decreasing and \eqref{alt:large} holds.
    \medskip

    Note that $xv_1 > xv_2 > \dots > xv_\ell >  \max \{v_{i}v_{\ell} \colon {1}\leq i< {\ell}\} =
    \max \{v_{i}v_{j} \colon {1}\leq i<j\leq {\ell}\}$, 
    where the last equality follows as in any canonical edge-ordering of $K_\ell$ the largest edge is incident to $v_{\ell}$. Thus, 
    $K_{n+1}[V(\widetilde K)\cup\{x\}]$ is a \sceo{} copy of~$K_{\ell +1}$ with special vertex~$x$, and with larger decreasing ordering.

    \medskip    
    \item \label{case:decsmall} The sequence $\{xv_i\}_{i\in[\ell]}$ is decreasing and \eqref{alt:small} holds.
    \medskip

    Note that $xv_\ell < \dots < xv_1 < \min \{v_{1}v_{i} \colon {1}< i \leq {\ell}\}= \min \{v_{i}v_{j} \colon 1\leq i<j\leq \ell\}$, where the last equality follows as in any cannonical edge-ordering of $K_\ell$ the smallest edge is incident to $v_{1}$. Thus, 
    $K_{n+1}[V(\widetilde K)\cup\{x\}]$ is a \sceo{} copy of~$K_{\ell +1}$ with special vertex~$x$, and with smaller decreasing ordering.

    \medskip
    \item \label{case:decmiddle} The sequence $\{xv_i\}_{i\in[\ell]}$ is decreasing and \eqref{alt:middle} holds.
    \medskip

As $xv_1> xv_2 > \dots > x v_{\ell}$, \eqref{alt:middle} implies that $xv_1>v_1v_2> xv_2$ and also 
$x v_2 > v_2 v_j > x v_j$ for all $3 \leq  j\leq \ell$.
Thus, $v_1v_2 > \max\{v_2v_i \colon 2<i\leq \ell\}$.
 Note though, however $\widetilde K$ is canonically ordered, we must 
 have that  $\max\{v_2v_i \colon 2<i\leq \ell\} > v_1v_2$. 
    Since this is a contradiction, this case cannot happen. 

    \medskip
    \item The sequence $\{xv_i\}_{i\in[\ell]}$ is increasing and \eqref{alt:middle} holds.
    \medskip

    In this case notice that for all $k \in [\ell -1]$ we have 
\begin{align}\label{eq:6a}
       x v_k < v_k v_{k+1} < x v_{k+1}.
    \end{align}
    Furthermore, 
    \begin{align}
        \begin{split}\label{eq:6}
            \max\{v_iv_k \colon 1\leq i < k\} < &xv_k  \text{ for all } 2 \leq k \leq \ell \ \ \text{ and } \\
            &xv_k < \min\{v_kv_i\colon k<i\leq \ell\} \text{ for all } k \in [\ell -1].
        \end{split}
    \end{align}

    When $\widetilde K$ is an inverse min (resp. inverse max) ordering, (\ref{eq:6a}) and (\ref{eq:6}) imply that
    $\{x, v_1,\dots, v_{\ell}\}$ induces a canonical ordering of the same type as~$\widetilde K$, with~$x$ as the last (resp. first) vertex.

    When $\widetilde K$ is a min ordering, we have
    $$v_iv_\ell 
    \overset{\phantom{\eqref{eq:6}}}{<}
    v_{i+1}v_{i+2} 
    \overset{\eqref{eq:6a}}{<}
    xv_{i+2} 
    \overset{\eqref{eq:6}}{<}
    v_{i+2}v_{i+4}\, ,$$
    where the first inequality follows by Remark~\ref{rem:con}.
    Since~$\ell=f^2-4f+5 \geq 2f-3$ for $f\geq 4$, restricting to the vertices of odd index in $\widetilde K$, we obtain from Remark~\ref{rem:middle} that $K_{n+1}[\{x,v_1,v_3,\dots, v_{2f-3}\}]$ is a \sceo{} copy of $K_{f}$ with special vertex $x$, and with middle increasing ordering.
    
    For the max ordering, we use an analogous argument: \eqref{eq:6} implies $v_{i}v_{i+2} < xv_{i+2} < v_{i+2}v_{i+3} < v_1v_{i+4}$. 
    Using again Remark~\ref{rem:middle} we have that $K_{n+1}[\{x,v_1,v_3,\dots, v_{2f-3}\}]$ is a \sceo{} copy of $K_{f}$ with special vertex $x$, and with middle increasing ordering.

    \medskip   
    \item  \label{case:inclarge} The sequence $\{xv_i\}_{i\in[\ell]}$  is increasing and \eqref{alt:large} holds. 
    \medskip
    
    We separate the proof of this case into three claims.

        \begin{claim}\label{claim:max}
        If $\widetilde K$ is a max or an inverse max ordering then~$K_{n+1}[V(\widetilde K)\cup\{x\}]$ contains a \sceo{} copy of~$K_f$ with larger increasing ordering and special vertex~$x$.
        \end{claim}
      \begin{claimproof}
        For these canonical orderings we have~$v_1v_\ell > \max \{v_{i}v_{j} \colon {1}\leq j\leq {\ell-1}\}$.
        Then, due to~\eqref{alt:large}, we have
        $$xv_\ell>\dots> xv_2 > xv_1 > v_1v_\ell >
        \max \{v_{i}v_{j} \colon {1}\leq i<j\leq {\ell-1}\}\,,$$
        and therefore~$\{x,v_1, \dots, v_{\ell-1}\}$ induces a larger increasing ordering.
        \end{claimproof}

        
        When $\widetilde K$ is a min or an inverse min ordering we will use the following claim. 
        
   \begin{claim}\label{claim:set}
    Suppose $\widetilde K$ is a min or an inverse min ordering.
    Either~$K_{n+1}[V(\widetilde K) \cup \{x\}]$ contains a larger increasing \sco{}  of~$K_f$ containing~$x$ or the following statement holds. 
    There is a set~$U_{f-3}\subseteq V(\widetilde K)$ such that, after relabeling the vertices, we have~$U_{f-3}:=\{v_1,\dots, v_{{f-1}}\}$ and, for all $i < f-2$,
        \begin{align}\label{eq:goal2}
            \max \{ v_i v_j \colon {i}< j\leq f-1\}
            < xv_i < 
            \min \{v_{j}v_{k} \colon {i}< j<k\leq {{f-1}}\}\,.
        \end{align}
\end{claim}              
      \begin{claimproof}
      Suppose~$\widetilde K$ is a min or an inverse min ordering and~$K_{n+1} [V(\widetilde K) \cup \{x\}]$ does not contain a larger increasing \sco{} of~$K_f$ containing~$x$. 
      For each $0 \leq r \leq f-3$, define $\ell _r:= \ell -r(f-2)$; so $\ell_0 = \ell$ and 
      $$\ell_{f-3} = \ell - (f-3)(f-2) = (f^2-4f+5) - (f-3)(f-2) = f-1\,.$$
      To prove the claim we proceed iteratively as follows.
        Suppose for some~$0\leq r<f-3$ there is a set of vertices $U_r:=\{v_1, \dots, v_{\ell_r}\}$ satisfying 
        \begin{align}\label{eq:goal}
          \max \{ v_i v_j \colon {i}< j\leq \ell _r\}
            < xv_i < 
            \min \{v_{j}v_{k} \colon {i}< j<k\leq {\ell_{r}}\}\, \text{ for all~$i\leq r$} .
        \end{align}
        We shall find a set $U_{r+1}\subseteq U_{r}$ such that, after relabeling, we have $U_{r+1}:=\{v_1, \dots, v_{\ell_{r+1}}\}$ and where~\eqref{eq:goal} holds for~$r+1$ instead of~$r$. 
        To start the iteration take~$r=0$ and let~$U_0:=V(\widetilde K)$.
        
        If $xv_{r+1} >\max \{v_{j}v_{k} \colon {r}< j<k< {r+f}\}$, then, since~$\{xv_j\}_{j\in [\ell]}$ is increasing, we have $$xv_{r+f-1}>xv_{r+f-2}>\dots >xv_{r+1}> \max \{v_{j}v_{k} \colon {r}< j<k< {r+f}\}\,.$$
        Thus,~$\{x,v_{r+1}, \dots, v_{r+f-1}\}$ induces a larger increasing \sco{} of~$K_f$ contradicting our initial supposition.
        So we may assume that~$xv_{r+1} <
        \max \{v_{j}v_{k} \colon {r}< j<k< {r+f}\}$ and conclude
        \begin{align}\label{eq:r+1}
        \begin{split}
            \max\{v_{r+1}v_i\colon r+1<i\leq\ell_r\}
            \overset{\eqref{alt:large}}{<}
            xv_{r+1}
            &\overset{\phantom{\eqref{alt:large}}}{<}
            \max \{v_{j}v_{k} \colon {r}< j<k< {r+f}\} \\
            &\overset{\phantom{\eqref{alt:large}}}{<}
            \min \{v_{j}v_{k} \colon {r+f}\leq j<k\leq {\ell_r}\}\,.
        \end{split}
        \end{align}
        The last inequality follows from the fact that $\widetilde K$ is min or inverse min  ordered and by recalling Remark~\ref{rem:con}. 
        
        Delete the vertices $v_{r+2}, \dots, v_{r+f-1}$, relabel the remaining vertices, and let~$U_{r+1}:=\{v_1, \dots, v_{\ell_{r+1}}\}$ be the set of vertices after the deletion and the relabeling.
        We shall prove that~$U_{r+1}$ satisfies \eqref{eq:goal} for $r+1$ instead of~$r$.
        First, observe that for~$i\leq r+1$,~$v_i$ is not deleted and keeps the same label as in~$U_r$.
        Moreover, since we only delete vertices, the sets from which we take the maximum and minimum in~\eqref{eq:goal} are now smaller, and thus, for~$i\leq r$, \eqref{eq:goal} becomes in fact less restrictive after the deletion and relabeling.
        Therefore, the inequalities in~\eqref{eq:goal} still hold for $i\leq r$ with~$\ell_{r+1}$ instead of~$\ell_r$. 
        We still need to prove that they hold for~$i=r+1$. 
        For that, note that vertex~$v_{r+f}$ is relabeled as~$v_{r+2}$ in~$U_{r+1}$ and therefore~\eqref{eq:r+1} implies
        $$\max\{v_{r+1}v_i\colon r+1<i\leq\ell_{r+1}\}
            <
            xv_{r+1}
           <
           \min \{v_{j}v_{k} \colon {r+2}\leq j<k\leq {\ell_{r+1}}\}\,,$$
        in~$U_{r+1}$.
        That is, the inequalities in \eqref{eq:goal} hold for $i=r+1$ in~$U_{r+1}$ and with $\ell_{r+1}$ instead of~$\ell_{r}$.
        Hence,~\eqref{eq:goal} holds for~$r+1$ instead of $r$. 

        Since~$\ell_{f-3} = f-1$, after~$f-3$ steps we obtain~$U_{f-3}=\{v_1,\dots, v_{f-1}\}$ satisfying~\eqref{eq:goal2}  for every~$i<f-2$.
\end{claimproof}
        

We use Claim~\ref{claim:set} to prove the following claim finishing the proof of this case. 

    \begin{claim} \label{claim:min}
        Suppose $\widetilde K$ is a min or an inverse min ordering.
        Either~$K_{n+1}[V(\widetilde K) \cup \{x\}]$ contains a larger increasing \sco{} of~$K_f$ containing~$x$ or the following two statements hold.
    \begin{itemize}
        \item If~$\widetilde K$ is a min canonical ordering then~$K_{n+1}[V(\widetilde K)\cup\{x\}]$ contains a min canonically ordered copy of~$K_f$ containing~$x$. 
        \item  If~$\widetilde K$ is an inverse min canonical ordering then~$K_{n+1}[V(\widetilde K)\cup\{x\}]$ contains a middle increasing \sco{} of~$K_f$ with special vertex~$x$. 
    \end{itemize}
    \end{claim}    

    \begin{claimproof}
        Suppose $\widetilde K$ is a min or an inverse min ordering and~$K_{n+1}[V(\widetilde K) \cup \{x\}]$ does not contain a larger increasing \sco{} of~$K_f$ containing~$x$. 
        Apply Claim~\ref{claim:set} to obtain a set~$U$ such that after relabeling the vertices we have $U:=\{v_1,\dots, v_{{f-1}}\}$ satisfying~\eqref{eq:goal2} for every~$i<f-2$.
        
        Since the sequence $\{xv_i\}_{i\in [\ell]}$ is increasing and because of~\eqref{alt:large} we deduce
        \begin{align}\label{eq:goaltail}
            v_{f-2}v_{f-1}<v_{f-2}x<v_{f-1}x .
        \end{align}
        If~$\widetilde K$ is a min canonical ordering then \eqref{eq:goal2} becomes~$v_iv_{f-1} < xv_i < v_{i+1}v_{i+2}$ for~$i<f-2$.
        Then, using~\eqref{eq:goaltail} it is easy to check that $U\cup \{x\}$ induces a min canonical ordering, with~$x$ playing the role of the last vertex~$v_f$. 
        If~$\widetilde K$ is an inverse min canonical ordering, then \eqref{eq:goal2} becomes~$v_iv_{i+1} < xv_i < v_{i+1}v_{f-1}$ for every~$i<f-2$.
        Then, we obtain from~\eqref{eq:goaltail} and Remark~\ref{rem:middle} that~$U\cup \{x\}$ induces a middle increasing ordering with special vertex~$x$.
    \end{claimproof}
        

    \medskip
    \item \label{case:incsmall} The sequence $\{xv_i\}_{i\in[\ell]}$ is increasing and \eqref{alt:small} holds.
    \medskip

    For this case we reverse the edge-ordering of $K_{n+1}[V(\widetilde K)\cup \{x\}]$ and the ordering of the vertices in the canonical part. 
    More precisely, let~$\back{K} := \back{K}_{n+1}[V(\widetilde K)\cup \{x\}]$ be the reverse of~$K_{n+1}[V(\widetilde K)\cup \{x\}]$ and let~$V(\back K)\setminus \{x\}$ be reordered as $V(\back K)\setminus \{x\}= \{v_1',\dots, v_\ell'\}$ where~$v_i':=v_{\ell-i+1}$.
    Then
        \smallskip
    \begin{enumerate}[label=\arlabel, wide]
        \item $\back{K}[V(\widetilde K)]$ is canonically ordered, \label{it:canonical}
        \item $\{xv_i'\}_{i\in [\ell]}$ is increasing, and \label{it:increase}
        \item \eqref{alt:large} holds for~$\back{K}$. \label{it:large}
    \end{enumerate}
      \smallskip
    Indeed, for~\ref{it:canonical} notice that the reverse of a canonical ordering is canonical after reversing the ordering of the vertices.
    For \ref{it:increase} observe that we reverse the ordering of the vertices and edges, so the sequence is still increasing.
    Finally,~\ref{it:large} is easy to deduce after noticing that~\eqref{alt:large} and~\eqref{alt:small} only depend on the ordering of the edges and not on the ordering of the vertices. 
    

    Observe that conditions~\ref{it:canonical}--\ref{it:large} are the same conditions we have for \hyperref[case:inclarge]{\textit{Case (5)}}.
    Thus, to address our current case, we apply Claims~\ref{claim:max} and \ref{claim:min} to the edge-ordered graph~$\back{K}$.
    
    More precisely, when~$\widetilde K$ is a  min ordering or an inverse min ordering, then~$\back{K}$ is a max or an inverse max ordering. 
    Therefore, Claim~\ref{claim:max} implies that~$\back{K}$ contains a~\sceo{} copy of~$K_f$ with larger increasing ordering and special vertex~$x$. 
    Hence, $K_{n+1}[V(\widetilde K)\cup \{x\}]$ contains a~\sceo{} copy of~$K_f$ with smaller increasing ordering and special vertex~$x$.

   By an analogous argument but using Claim~\ref{claim:min} instead of Claim~\ref{claim:max}, we have that if~$\widetilde K$ is a max ordering or an inverse max ordering then~$K_{n+1}[V(\widetilde K) \cup \{x\}]$ contains a \sco{} copy of $K_f$ containing~$x$. 
    Moreover, this copy of $K_f$ is either a smaller increasing ordering, a max canonical ordering, or a middle increasing ordering.   \hfill\qedhere
    \end{enumerate}
    \end{proof}

\section{Universally tileable graphs}\label{subsec:char2}

We begin this section with the proof of Theorem~\ref{thm:uni}.

\smallskip

{\it \noindent Proof of Theorem~\ref{thm:uni}.}
To prove the statement we will show that \ref{it:univtil} implies \ref{it:univturan}, \ref{it:univturan} implies \ref{it:univdescrip} and \ref{it:univdescrip}  implies \ref{it:univtil}.
If  an edge-ordered graph is tileable then by definition it is Tur\'anable.
Thus, \ref{it:univtil} immediately implies \ref{it:univturan}. One part of 
Theorem~2.18 from~\cite{gmnptv} precisely states that  \ref{it:univturan} is equivalent to \ref{it:univdescrip}.
It therefore remains to show that \ref{it:univdescrip} implies \ref{it:univtil}.

First assume that $H$ is a $K_3$ together with a (possibly empty) collection of isolated vertices. 
Note that all edge-orderings of $H$ are isomorphic, so every edge-ordering of $K_{|H|}$ contains a spanning copy of $H^{{\scaleto{{\leq}}{4.3pt}}}$, for every edge-ordering~$\leq$. 
Thus $H$ is universally tileable.

Now, suppose $H$ is a path on three edges. There are three types of edge-ordering of $H$: $123$, $132$, and $213$. The latter two are contained in any edge-ordering of $C_4$ and so are tileable. The former is just $P_3^{{\scaleto{{\leqslant}}{4.3pt}}}$, so is tileable by Theorem~\ref{Pkfactor}. Thus, $H$ is universally tileable. Note that adding isolated vertices to a tileable edge-ordered graph results in another tileable edge-ordered graph.
Therefore, every path on three edges together with a (possibly empty) collection of isolated vertices forms a universally tileable graph.

\smallskip

Finally, assume that $H$ is a star forest and $H^{\scaleto{{\leq}}{4.3pt}}$ is any edge-ordering of $H$. Let $h:=|H|$.
We now check that we can find a copy of $H^{\scaleto{{\leq}}{4.3pt}}$ in any \sceo\ $K_{h}$. 
As usual we write $\{x,v_1,\dots, v_{h-1}\}$ for the vertices of a \sceo\ $K_{h}$, where $x$ is the special vertex.

Given any vertex $v$ in $H^{\scaleto{{\leq}}{4.3pt}}$, 
$H^{\scaleto{{\leq}}{4.3pt}}-v$ is a star forest and so is Tur\'anable by \cite[Theorem 2.18]{gmnptv}; thus, by Theorem~\ref{thm:turanable}, any canonical ordering of $K_{h-1}$ contains a copy of 
$H^{\scaleto{{\leq}}{4.3pt}}-v$.

Consider any  smaller increasing/decreasing \sco\ of $K_{h}$. Let~$uw$  be the smallest edge in~$H^{\scaleto{{\leq}}{4.3pt}}$ where $u$ is a leaf of $H^{\scaleto{{\leq}}{4.3pt}}$ . By the remark in the previous paragraph, our edge-ordered $K_{h}$ contains a copy of $H^{\scaleto{{\leq}}{4.3pt}}-u$ that does not contain $x$. By definition of a smaller increasing/decreasing \sco, we can now add $x$ to this copy of $H^{\scaleto{{\leq}}{4.3pt}}-u$ to obtain a copy of $H^{\scaleto{{\leq}}{4.3pt}}$ in our edge-ordered $K_{h}$.
For a larger increasing/decreasing \sco\ of $K_{h}$ one can argue analogously, but take $uw$ to be the largest edge in~$H^{\scaleto{{\leq}}{4.3pt}}$, instead of the smallest. 


Next we consider the middle increasing \sco{} of $K_{h}$.
Let $\{K_{1,t_i}\}_{1\le i \le k}$ be the collection of $k$ stars that form the components of $H$ and let~$C\subseteq V(H)$ be the set of centers of these stars (if $t_i=1$, for the star $K_{1,t_i}$ we pick the center arbitrarily). 
Let~$L:=V(H)\setminus C$ and note that every vertex in~$L$ is a leaf.
We define an ordering of the leaves in $L$ as follows.
Given two leaves~$\ell,m\in L$ we write~$\ell < m$ if and only if~$\ell u < m w$ in $H^{\scaleto{{\leq}}{4.3pt}}$, where~$u,w\in C$ are the unique neighbors of~$\ell$ and $m$ in $H^{\scaleto{{\leq}}{4.3pt}}$ respectively (note $u$ and $w$ are not necessarily distinct).
Set~$L=:\{\ell_1,\dots, \ell_{|L|}\}$ where 
$\ell_1 <\dots < \ell_{|L|}$; note that~$\vert L\vert = \vert E(H)\vert$.

We are now ready to embed~$H^{\scaleto{{\leq}}{4.3pt}}$ into a middle increasing~\sco{} of~$K_{h}$. 
We first assume that the canonical part of~$K_{h}$ is a min or an inverse min ordering.
Then, using the labeling given in Definitions~\ref{def:canonical} and~\ref{def:starcanonical}, it is easy to check that for every~$1\leq i<j<k,m\leq h-1$, we have 
\begin{alignat}{3}\label{eq:universalmin}
        & v_iv_k  && \,<\,\, && v_jv_m,   \nonumber \\ 
         & v_ix  &&\,<\,\, && v_j x, \nonumber \\
        & v_iv_k  &&\,<\,\, && v_j x, \qquad\text{and} \\ 
        & v_ix    &&\,<\,\, && v_jv_k\,. \nonumber
\end{alignat}
We embed the vertices in~$L=\{\ell_1,\dots, \ell_{|L|}\}$ into the \sco{} of~$K_{h}$ as follows:
$$\ell_i \mapsto v_i \text{~for every~}i\in \big[|L|\big]\,.$$
We embed the vertices in~$C$ arbitrarily among the rest of the vertices in $K_h$.
We need to check that this embedding induces a copy of $H^{\scaleto{{\leq}}{4.3pt}}$ in our edge-ordered $K_h$. This is clearly the case though: if $e_1, e_2 \in E(H^{\scaleto{{\leq}}{4.3pt}})$ such that $e_1<e_2$ then $e_1$ is mapped to some edge $v_i y$ in $K_h$ and $e_2$ to some edge $v_j z$ in $K_h$, where $i<j\leq |L|$. Then (\ref{eq:universalmin}) implies that
$v_i y < v_j z$ in our edge-ordering of $K_h$.

If the canonical part of~$K_{h}$ is a max  or an inverse max ordering, then we proceed analogously.
In this case we embed the leaves in $L$ at the end of the \sco{} and the vertices in~$C$ at the beginning.
More precisely, we define the embedding so that
$$\ell_i \mapsto v_{\vert C\vert+i-1} \text{~for every~}i\in \big[|L|\big]\,,$$
and we embed the vertices in~$C$ arbitrarily among the rest of the vertices $K_h$.
Then similarly to before, this embedding induces a copy of 
$H^{\scaleto{{\leq}}{4.3pt}}$ in our edge-ordering of $K_h$.
\qed 

\smallskip

There are some cases where the solution of Question~\ref{ques1} is an easy consequence of known tiling results for (unordered) graphs. 
In particular, the next result solves this problem for all edge-orderings of connected universally tileable graphs.

\begin{prop}\label{propf}
{\color{white}a}
\begin{itemize}
    \item Let $K^{\scaleto{{\leqslant}}{4.3pt}}_3$ denote the edge-ordered version of $K_3$. Then
    $f(n,K^{\scaleto{{\leqslant}}{4.3pt}} _3)=2n/3$.
    \item Let $S$ denote an edge-ordered graph whose underlying graph is a star. Then $f(n,S)=n/2+O(1)$.
    \item Let $P:=132$. Then $f(n,P)= n/2+O(1)$.
    \item Let $P':=213$. Then $f(n,P')= n/2+O(1)$.
    \item Recall $P_3^{{\scaleto{{\leqslant}}{4.3pt}}}=123$. Then $f(n,P_3^{{\scaleto{{\leqslant}}{4.3pt}}})= n/2+o(n)$.
\end{itemize}

\end{prop}
\proof
The first part of the proposition follows immediately from the 
Corr\'adi--Hajnal theorem~\cite{cor}. 

Up to isomorphism, there is only one edge-ordering of a star on a given number of vertices. Thus, for any edge-ordered star $S$, the K\"uhn--Osthus theorem~\cite{kuhn2} implies that $f(n,S)=n/2+O(1)$. 

Any edge-ordering of $C_4$ contains a copy of the edge-ordered path $P=132$. 
The K\"uhn--Osthus theorem~\cite{kuhn2} implies that the minimum degree threshold for forcing a perfect $C_4$-tiling in an $n$-vertex graph $G$ is $n/2+O(1)$; so $f(n,P)\leq n/2+O(1)$.
Moreover, consider the $n$-vertex graph  consisting of two disjoint cliques $X$, $Y$ whose sizes are as equal as possible, under the constraint that
 $4$ does not divide $|X|$ or $|Y|$. Then every edge-ordering $G$ of this graph does not contain a 
perfect $P$-tiling and $\delta (G) \geq n/2 -2$. Thus,  $f(n,P)> n/2-2$ and so $f(n,P)= n/2+O(1)$.
The same argument shows that  $f(n,P')= n/2+O(1)$.
 Finally, in Theorem~\ref{Pkfactor} we saw that $f(n,P_3^{{\scaleto{{\leqslant}}{4.3pt}}})= (1/2+o(1))n$. 
\endproof

\section{Proof of  Theorem~\ref{hscorollary}}\label{subsec:hscor}

Let $G$ be an edge-ordered graph on $n\geq T(F)$ vertices with minimum degree  $\delta (G) \geq (1-\frac{1}{T(F)})n$, and so that $|F|$ divides $n$. Let $G'$ denote the underlying graph of $G$. When $T(F)$ divides $n$, we apply the Hajnal--Szemer\'{e}di theorem \cite{hs} to $G'$, to obtain an (unordered) perfect $K_{T(F)}$-tiling in $G'$. By the definition of $T(F)$, each edge-ordered copy of $K_{T(F)}$ in $G$ contains a perfect $F$-tiling. 
Thus, combining these tilings, we obtain a  perfect $F$-tiling in $G$. 

When $T(F)$ does not divide $n$, then $n=aT(F)+b$ for  $a,b \in \mathbb N$ such that $0<b<T(F)$. As $n$ and $T(F)$ are divisible by $|F|$,  we have that $b/|F| \in \mathbb N$.
Since ${b}/{T(F)}<1$, we must have that $\delta (G) \geq n-a =(1-\frac{1}{T(F)})(n-b)+b.$ 

We will now repeatedly remove disjoint copies of $F$ from $G$, until the resulting edge-ordered graph has its order divisible by $T(F)$.
Assume that we have already removed $c$ copies of $F$ from $G$, where $0\le c<{b}/{|F|}$;
then the remaining edge-ordered graph on $n-c|F|$ vertices has minimum degree at 
least 
$$\Big(1-\frac{1}{T(F)}\Big)(n-b)+b-c|F|\ge \Big(1-\frac{1}{T(F)}\Big)(n-c|F|)+(b-c|F|)\frac{1}{T(F)}.$$ 
This lower bound 
guarantees that an unordered $K_{T(F)}$ exists in the underlying graph; within the corresponding edge-ordered copy of $K_{T(F)}$ lying in $G$, we can find a copy of $F$. Thus, we may again remove a copy of $F$ and repeat this process.

This process ensures that we can  remove ${b}/{|F|}$ copies of $F$ from $G$. The resulting edge-ordered graph has $n-b$ vertices and  minimum degree at least $(1-\frac{1}{T(F)})(n-b)$. Since $T(F)$ divides $n-b$, as in the previous case this edge-ordered graph contains a perfect $F$-tiling; combining this tiling with our removed copies of $F$, we obtain  a perfect $F$-tiling in $G$, as desired.\qed

\section{Proof of Theorem~\ref{Pkfactor}}\label{sec:mainproof}

For the proof of Theorem~\ref{Pkfactor} we use the absorbing method, which divides the proof into two main parts:  finding an absorber and constructing an almost perfect $\Pk$-tiling. 

The following two subsections are devoted to the Absorbing Lemma (Lemma~\ref{lemma:globalabs}) and the Almost Perfect Tiling Lemma (Lemma~\ref{lemma:almosttiling}) respectively. 
We finish this section by combining these two results to give  the proof of Theorem~\ref{Pkfactor}.

\subsection{Absorbers}

Let $F$ be an edge-ordered graph. Given an edge-ordered graph $G$, a set $S \subseteq V(G)$ is an \emph{$F$-absorbing set for $Q \subseteq V(G)$}, if both
$G[S]$ and $G[S\cup Q]$ contain perfect $F$-tilings. 

To prove Theorem~\ref{Pkfactor}, we make use of the following, now standard, absorbing lemma.
\begin{lemma}\label{lo}
Let $f,s\in \mathbb N$ and $\xi >0$. Suppose that $F$ is an edge-ordered graph on $f$ vertices. Then there exists an $n_0 \in \mathbb N$ such that the following holds. Suppose that $G$ is an edge-ordered graph
on $n \geq n_0$ vertices so that, for any $x,y \in V(G)$, there are at least $\xi n^{sf-1}$ $(sf-1)$-sets $X \subseteq V(G)$ such that both $G[X \cup \{x\}]$ and $G[X \cup \{y\}]$ contain perfect $F$-tilings.
Then $V(G)$ contains a set $M$ so that
\begin{itemize}
\item $|M|\leq (\xi/2)^f n/4$;
\item $M$ is an $F$-absorbing set for any $W \subseteq V(G) \setminus M$ such that $|W|\leq (\xi /2)^{2f} n/(32s^2 f^3)$ and~$|W| \in f \mathbb N$. \qed
\end{itemize}
\end{lemma} 
Lemma~\ref{lo} was proven by Lo and Markstr\"om~\cite[Lemma 1.1]{lo} in the case when $G$ is an unordered graph.
However, the proof in the edge-ordered setting is identical (so we do not provide a proof here). 

As mentioned in the introduction, R\"odl~\cite{rodl} proved that every edge-ordered graph on~$n$ vertices with at least~$k(k+1)n/2$ edges contains a monotone path of length $k$. 
Here we will need the following supersaturated version of this  result.

\begin{lemma}[Supersaturation Lemma]\label{lemma:Rodl}
Let~$k\in \mathbb N$ and~$\zeta>0$. Then there exists an $n_0 \in \mathbb N$ such that the following holds for every~$n\geq n_0$.
Every~$n$-vertex edge-ordered graph $G$ with at least~$\zeta n^2$ edges contains at least~$\zeta^k\,{2^{-k^2}}n^{k+1}$ copies of~$\Pk$.
\end{lemma}

\begin{proof}
The proof goes by induction on~$k$. The case $k=1$ is trivial. 
Suppose the statement is true for~$k-1$, and take~$n_0$ large enough to apply the induction hypothesis for~$\zeta/2$.

Let $G$ be an $n$-vertex edge-ordered graph  as in the statement of the lemma.
For every vertex~$v\in V(G)$ delete the last~$\min\{d(v), \zeta n/2\}$ edges (under the total order) that are incident to~$v$.
Let~$\tilde G$ denote the resulting edge-ordered graph. 
Since~$e(\tilde G)\geq \zeta n^2 - \zeta n^2/2 =\zeta n^2/2$, by the induction hypothesis we have that~$\tilde G$ contains at least~
$$\Big(\frac{\zeta}{2}\Big)^{k-1}\!\!\cdot {2^{-(k-1)^2}} n^{k} = \zeta^{k-1}\, 2^{-(k-1)^2-(k-1)}n^{k}$$ copies of~$\Pkkk$. 
Fix one such copy $P = v_1 \cdots v_{k}$ and observe that, since~$d_{\tilde G}(v_{k})>0$, $\zeta n/2$ edges incident to~$v_{k}$ were deleted from $G$ that are all larger than~$v_{k-1}v_{k}$ in the total order of~$E(G)$.
Moreover, at most~$k-1$ of them are incident to a vertex in $P$, which implies that at least~$\zeta n/2-(k-1)\geq \zeta n/4$ of them, combined with $P$, form a copy of~$\Pk$ in~$G$.
Therefore, we obtain at least~$$\frac{\zeta^{k-1}}{2^{(k-1)^2+(k-1)}}\, n^{k} \cdot \frac{\zeta}{4}\,n \geq \frac{\zeta^{k}}{2^{k^2}}\, n^{k+1} $$ copies of~$\Pk$ in~$G$.
\end{proof}

Note the proof of Lemma~\ref{lemma:Rodl} really uses that the path we consider is monotone. Indeed, the inductive step our proof requires that given an edge-ordered path $P$, we add an edge $e$ larger than all those edges in $P$, and that $e$ is incident to the largest edge currently in $P$.

In order to apply Lemma~\ref{lo} we introduce the following notion.


\begin{definition}[Local Absorbers]\label{def:abs} \rm
Let~$x, y\in V(G)$ be  distinct vertices of an edge-ordered graph~$G$.
Let~$P_x, P_y\in \binom{V(G)}{k}$ be disjoint and~$w\in V(G)\setminus (P_x\cup P_y\cup \{x,y\})$ so that 
$ x \notin P_y$ and $y \notin P_x$.
We say that the set
$$A := P_x\,\cup\, P_y\,\cup\, \{w\}$$
is a~\textit{$\Pk$-local-absorber for $x$ and $y$} if 
\begin{enumerate}
    \item $G[\{x\} \cup P_x]$ and~$G[\{w\}\cup P_x]$ contain spanning copies of~$\Pk$ and
    \item $G[\{y\} \cup P_y]$ and~$G[\{w\}\cup P_y]$ contain spanning copies of~$\Pk$.
\end{enumerate}
\end{definition}

Observe that if~$A$ is a~$\Pk$-local-absorber for~$x$ and $y$ then both~$G[A\cup \{x\}]$ and $G[A\cup \{y\}]$ contain perfect~$\Pk$-tilings. 
That is,~$A$ can play the role of~$X$ in Lemma~\ref{lo} with $s=2$.
The following lemma allows us to find many local absorbers for every pair of vertices~$x,y\in V(G)$.

\begin{lemma}\label{lemma:localabs}
For every~$k\in \mathbb N$ and for every~$0<\eta<1/2$ there is a~$\xi>0$ and an~$n_0\in \mathbb N$ such that the following holds for every $n\geq n_0$.
Let~$G$ be an $n$-vertex edge-ordered graph with~$\delta(G)\geq (1/2+\eta)n$. Then for every two vertices~$x,y\in V(G)$ there are at least~$\xi n^{2k+1}$ $\Pk$-local-absorbers for~$x$ and $y$. 
\end{lemma}

\begin{proof}
Given~$k\in \mathbb N$ and~$\eta>0$ let~$$\zeta := \frac{\eta^{k}}{2^{k^2+4k}} \qquad \text{and} \qquad \xi := \frac{\eta \zeta^2}{16 (2k+1)!}\,,$$
and suppose~$n_0\in \mathbb N$ is sufficiently large. 
Let~$G$ be as in the statement of the lemma.

For every~$x\in V(G)$ define 
$$\mathcal P_x := \Big\{P\in \binom{V(G)}{k} \colon G[\{x\}\cup P]\text{ contains a copy of~$\Pk$}\Big\}\,.$$
We first show that there is a subset~$\mathcal P_x'\subseteq \mathcal P_x$ of size at least~$\zeta n^{k}/2$ such that for every~$P\in \mathcal P_x'$ there is a set~$W_x(P)\subseteq V(G)\setminus P$ satisfying
\smallskip
\begin{enumerate}[label={(\roman*)}]
    \item \label{it:neighbourhood} $P\in \mathcal P_w$ for every~$w\in W_x(P)$ and
    \item \label{it:size} $\vert W_x(P)\vert \geq \big(\tfrac{1}{2}+\tfrac{\eta}{4}\big)n$.
\end{enumerate}
\smallskip

In order to do this, we partition $N(x) = L(x)\dot\cup S(x)$ as follows.
We say a vertex~$u\in N(x)$ is \textit{large} if the set~$\{v\in N(u)\colon xu < vu\}$ is of size at least~$\eta n/2$. 
Otherwise, we say~$u$ is \textit{small}.
Let~$L(x)$ and~$S(x)$ denote the set of large and small vertices in~$N(x)$, respectively.
Notice that if~$u$ is small then the set~$\{v\in N(u) \colon xu > vu\}$ is of size at least~$\eta n/2$ (and actually, at least of size~$n/2$). 
Assume that~$|L(x)| \geq \vert N(x)\vert /2 \geq n/4$; the case~$|S(x)|\geq n/4$ is analogous.

For every vertex~$u\in L(x)$, let~$E(u)$ be the set of the last~$\eta n/2$ edges incident to $u$ in the total order of $E(G)$. 
Since $u$ is large, all edges in~$E(u)$ are larger than~$xu$. 
For~$E_x := \bigcup_{u\in L(x)} E(u)$, consider the subgraph~$\widetilde G:=(V(G), E_x)\subseteq G$. Note that~$\vert E_x\vert \geq \eta n^2/16$. 
Thus, Lemma~\ref{lemma:Rodl} implies that~$\widetilde G$ contains at least~$\zeta n^{k+1}$ monotone paths of length $k$.
Since every edge in~$E_x$ is incident to a vertex in~$L(x)$, by dropping the first or the last vertex in each path, we obtain at least~$\zeta n^{k}/2$ monotone paths of length $k-1$ in $\widetilde G$ starting with a vertex in~$L(x)$.
That is, the set 
$$\mathcal P'_x := \Big\{P\in \binom{V(G)}{k} \colon \widetilde G[P]\text{ contains a copy of~$P_{k-1}^{\scaleto{\leqslant}{4.3pt}}$ starting with a vertex in~$L(x)$}\Big\}\,$$
is of size at least~$\zeta n^{k}/2$.
Moreover, notice that~$\mathcal P_x'\subseteq \mathcal P_x$.
Indeed, let~$u_1\cdots u_k$ be a monotone path with~$P=\{u_1,\dots, u_k\}\in \mathcal P_x'$.
Since~$u_1\in L(x)$, we have~$xu_1 < u_1u_2$, and therefore~$G[\{x\}\cup P]$ contains a copy of~$\Pk$, meaning that~$P\in \mathcal P_x$.  
Now, we shall prove that for every~$P\in\mathcal P_x'$ there is a set~$W_x(P)$ satisfying~\ref{it:neighbourhood} and~\ref{it:size}.

Consider some $P=\{u_1,\dots, u_k\}\in \mathcal P_x'$ where $u_1u_2$ is the first edge of the copy of
$P_{k-1}^{\scaleto{\leqslant}{4.3pt}}$ in $\widetilde G[P]$.
Let $N'(u_1)$ denote the set of vertices $w$ in $N(u_1)$ such that $u_1w \not \in E(u_1) $.
Define $W_x(P):=N'(u_1) \setminus  P$.
Thus, since~$u_1u_2\in E(u_1)$, for~$w\in W_x(P)$ we have~$wu_1 < u_1u_2$ which means that~$W_x(P)$ satisfies condition~\ref{it:neighbourhood}. 
Condition~\ref{it:size} follows as $\delta(G)\geq (1/2+\eta)n$ and~$\vert E(u_1)\vert = {\eta n}/{2}$.

\smallskip

Finally, given~$x,y\in V(G)$ consider~$\mathcal P_x'$ and~$\mathcal P_y'$. 
Observe that the number of pairs~$(P_x,P_y)\in \mathcal P_x'\times \mathcal P_y'$ such that~$\vert P_x\cap P_y\vert \geq 1$ is at most~$k^2n^{2k-1}$ and therefore, since~$n$ is sufficiently large, there are at least
$$\frac{\vert \mathcal P_x'\times \mathcal P_y' \vert }{2} \geq \frac{\zeta^2 n^{2k}}{8}$$ 
disjoint pairs in $\mathcal P_x'\times \mathcal P_y'$.
Given a disjoint pair~$(P_x,P_y)\in\mathcal P_x'\times\mathcal P_y'$ and a vertex~$w\in W_x(P_x)\cap W_y(P_y)$, it is easy to see that~$A:=P_x\cup P_y\cup \{w\}$ is a $\Pk$-local-absorber for~$x$ and $y$. 
Because of~\ref{it:size}, $\vert W_x(P_x)\cap W_y(P_y)\vert \geq \eta n/2$, and therefore, there are at least
$$\frac{\zeta^2 n^{2k}}{8}\cdot \frac{\eta n}{2} \cdot \frac{1}{(2k+1)!} = \xi n^{2k+1}$$
$\Pk$-local-absorbers for~$x$ and~$y$. 
In particular, we divide by $(2k+1)!$ as the same $\Pk$-local-absorber $A$ 
arises from at most $(2k+1)!$  tuples $(P_x,P_y,w)$.
\end{proof}

The Absorbing Lemma is now an immediate consequence of Lemmas~\ref{lo} and~\ref{lemma:localabs}. 

\begin{lemma}[Absorbing Lemma]\label{lemma:globalabs}
For every~$k\in \mathbb N$ and~$\eta>0$ there is~$0<\xi<\eta$ and  an~$n_0\in \mathbb N$ such that the following holds for every $n\geq n_0$. 
If $G$ is an edge-ordered graph on $n$ vertices with~$\delta(G)\geq (1/2+\eta)n$, then there is a set~$M\subseteq V(G)$ of size at most~$\xi n$ which is a~$\Pk$-absorbing set for every  $W \subseteq V(G) \setminus M$ such that $|W| \in (k+1) \mathbb N$ and  $|W|\leq {\xi^{3} n}$.\qed
\end{lemma}

\subsection{Almost perfect tilings}\label{subsec:almost}

Given an (unordered) graph~$F$, Koml\'os \cite{Komlos} established an asymptotically optimal minimum degree condition  that forces  a graph~$G$ to contain an $F$-tiling covering all but at most~$o(n)$ vertices.
To present this result, we need to introduce the following parameter. 
Given a graph~$F$, the \textit{critical chromatic number $\chi_{cr}(F)$ of~$F$} is defined as
$$\chi_{cr}(F) := (\chi(F)-1)\frac{\vert V(F)\vert}{\vert V(F)\vert-\sigma(F)}\,,$$
where~$\chi(F)$ is the chromatic number of~$F$ and~$\sigma(F)$ denotes the size of the smallest possible color class in any $\chi(F)$-coloring of $F$.

\begin{theorem}[\cite{Komlos}]\label{thm:Komlos}
For every~$\eps>0$ and every graph~$F$, there is an~$n_0\in \mathbb N$ such that the following holds for every~$n\geq n_0$.
If~$G$ is a graph on~$n$ vertices with
$$\delta(G)\geq \Big(1-\frac{1}{\chi_{cr}(F)}\Big)n\, ,$$
then~$G$ contains an~$F$-tiling covering at least~$(1-\eps)n$ vertices. 
\end{theorem}

Theorem~\ref{thm:Komlos} is best possible in the following sense: given any graph $F$ and any $\gamma<1-\frac{1}{\chi_{cr}(F)}$, there exist $\eps >0$ and  $n_0 \in \mathbb N$ so that if $n \geq n_0$ there is an $n$-vertex graph $G$ with $\delta (G) \geq \gamma n$
that does not contain an $F$-tiling covering at least~$(1-\eps)n$ vertices.

For the (unordered) path~$P_k$ of length~$k$, Theorem~\ref{thm:Komlos} ensures the existence of an almost perfect $P_k$-tiling in every~$n$-vertex graph with minimum degree~$\delta(G)\geq n/2$ when $k$ is odd and~$\delta(G)\geq kn/(2k+2)$ when~$k$ is even. 
The following lemma says that the same minimum degree condition ensures an almost perfect~$\Pk$-tiling in an edge-ordered graph $G$.

\begin{lemma}[Almost Perfect Tiling Lemma]\label{lemma:almosttiling}
Let~$k\in \mathbb N$ and~$\eps>0$. There is an~$n_0\in \mathbb N$ such that the following holds for every~$n\geq n_0$. 
Let~$G$ be an $n$-vertex edge-ordered graph with
\[
    \delta(G) \geq 
    \begin{cases}
        \frac{n}{2} &\text{ if $k$ is odd}\\
        \frac{kn}{2k+2} &\text{ if $k$ is even}\,.
    \end{cases}
\]
Then, $G$ contains a $\Pk$-tiling covering at least $(1-\eps)n$ vertices.
\end{lemma}
The same example that shows Theorem~\ref{thm:Komlos} is best possible for $P_k$ shows that Lemma~\ref{lemma:almosttiling} is best possible for $\Pk$. More precisely, if
$k$ is odd consider any $0<\gamma <1/2$ and set $\eps:=1/2-\gamma$; if $k$ is even consider any $0<\gamma <k/(2k+2)$ and set $\eps:=k/(2k+2)-\gamma$. 
Let $G$ be any edge-ordering of the complete bipartite graph with vertex classes of size $\gamma n$ and $(1-\gamma)n$. Then $\delta (G)= \gamma n$ and 
$G$ does not contain a $\Pk$-tiling covering more than $(1-\eps)n$ vertices.
\begin{proof}[Proof of Lemma~\ref{lemma:almosttiling}]
Given~$k\in \mathbb N$ and $\eps>0$, let~$\zeta := \frac{(k+1)^2-1}{4(k+1)^2}$ and let~$n_1\in \mathbb N$ be the~$n_0$ given by Lemma~\ref{lemma:Rodl} for $k+1$ instead of~$k$.
Moreover, let~$m\geq \frac{2n_1}{\eps (k+1)}$ and suppose~$n_0$ is sufficiently large with respect to all other constants.
Finally, let~$G$ be as in the statement of the lemma.

Set $a:=\lceil (k+1)/2 \rceil$ and $b:=\lfloor (k+1)/2 \rfloor$, and notice that~$\chi_{cr}(P_k)=\chi_{cr}(K_{am,bm})$.
Therefore, applying Theorem \ref{thm:Komlos} (to the underlying graph of $G$) we obtain a $K_{am,bm}$-tiling covering at least $(1-\eps/2)n$ vertices.
We shall prove that in each~$K_{am,bm}$ there is a~$\Pk$-tiling covering all but at most~$n_1$ vertices.
Observe that, for every positive integer~$t\in \mathbb N$, we have
\begin{align}\label{eq:edgesofKab}
    \vert E(K_{at, bt})\vert \geq \frac{(k+1)^2-1}{4}t^2 = \zeta (k+1)^2t^2 = \zeta \vert V(K_{at,bt})\vert^2\,.
\end{align}
Moreover,~$\vert V(K_{am,bm})\vert = (a+b)m = (k+1)m\geq n_1$, and therefore we may apply Lemma~\ref{lemma:Rodl}.
In fact, we will apply Lemma~\ref{lemma:Rodl} iteratively to find the desired~$\Pk$-tiling in~$K_{am,bm}$. 

If $k$ is even, then we apply Lemma~\ref{lemma:Rodl} to find a copy of~$\Pkk$ in~$K_{am,bm}$. 
After deleting one vertex, we get a copy of~$\Pk$ with exactly~$a=(k+2)/2$ vertices in the class of size~$am$. 
If~$k$ is odd, then we apply Lemma~\ref{lemma:Rodl} to obtain a copy of~$\Pk$, which must contain exactly~$a=(k+1)/2$ vertices in the class of size~$am$. 
In both cases, removing this copy of $P_k$ from $K_{am,bm}$ results in a copy of
$ K_{a(m-1),b(m-1)}$. 
Thus, since \eqref{eq:edgesofKab} holds for every~$t\in \mathbb N$, we may iteratively apply Lemma~\ref{lemma:Rodl} to find vertex-disjoint copies of~$\Pk$ in~$K_{am,bm}$ until there are at most~$n_1$ vertices left (in each~$K_{am,bm}$).

The initial~$K_{am,bm}$-tiling has at most~$n/\vert V(K_{am,bm})\vert = n/(m(k+1))$ copies of~$K_{am,bm}$ covering at least~$(1-\eps/2)n$ vertices in~$G$. 
Each of these copies of~$K_{am,bm}$ has a~$\Pk$-tiling covering all but at most $n_1$ vertices.
Therefore, there is a~$\Pk$-tiling in~$G$ covering all but at most
\begin{align*}
    \frac{\eps n}{2} + \frac{n}{m(k+1)}\,n_1
    \leq \eps\,n\,
\end{align*}
vertices, where the last inequality follows as~$\frac{n_1}{m(k+1)}\leq \tfrac{\eps}{2}$.
\end{proof}

\subsection{Proof of Theorem~\ref{Pkfactor}}

To prove the `moreover' part, given any 
$n \in \mathbb N$ divisible by $k+1$, let $G_0$ be an $n$-vertex edge-ordered graph consisting of two disjoint cliques whose sizes are as equal as possible under the constraint that neither has size divisible 
by $k+1$. Thus, $G_0$ does not contain a perfect~$\Pk$-tiling and~$\delta(G_0)\geq \lfloor n/2\rfloor-2$.

Given~$k\in \mathbb N$ and~$\eta>0$, let~$0<\xi<\eta$ be given by Lemma~\ref{lemma:globalabs}.
Let~$n_0\in \mathbb N$  be sufficiently large and let~$G$ be as in the statement of the theorem. 
Lemma~\ref{lemma:globalabs} yields a set $M\subseteq V(G)$ of size at most~$\xi n \leq \eta n$ which is a $\Pk$-absorbing set for every $W\subseteq V(G)\setminus M$ such that~$W\in (k+1)\mathbb N$ and~$\vert W\vert \leq \xi ^3 n$. 
As~$\delta (G\setminus M)\geq n/2+\eta n - \xi n\geq n/2$,
 Lemma~\ref{lemma:almosttiling} implies $G\setminus M$ contains a~$\Pk$-tiling $\mathcal T_1$ covering all but at most~$\xi ^3 n$ vertices.
Let~$L$ denote the set of vertices not covered by this tiling; notice that as~$\vert G\vert$ and $|M|$ are divisible by $k+1$, so is~$\vert L\vert$. By definition of $M$,
$G[M \cup L]$ contains a perfect~$\Pk$-tiling $\mathcal T_2$. Thus,
$ \mathcal T_1 \cup \mathcal T_2$ is a perfect~$\Pk$-tiling in~$G$.\qed

\begin{remark}\rm
Recall that, for $k \geq 4$, there is always an edge-ordering of $P_k$ that is not 
tileable.\footnote{This follows since neither of the edge-ordered paths $1423$ and $2314$ are Tur\'anable~\cite[Proposition 2.10]{gmnptv}.}
It would, however, be 
interesting to determine which edge-orderings of 
$P_k$ one can extend Theorem~\ref{Pkfactor} to cover.
Notice  that our proof of Theorem~\ref{Pkfactor} is tailored to monotone paths though.

Indeed, the proof of Lemma~\ref{lemma:almosttiling} uses Lemma~\ref{lemma:Rodl}, whose proof is specific to monotone paths~$\Pk$. Further, in the proof of
Lemma~\ref{lemma:localabs}, we use the fact that if~$P=u_1\cdots u_{k+1}$ is a monotone path, then~$u_1\cdots u_k$ is isomorphic to $u_2\cdots u_{k+1}$. 
In other words, the path obtained by dropping the last vertex is isomorphic to the one obtained by dropping the first one. 
It is not hard to see that this property is satisfied only by monotone paths.

In a forthcoming paper, the second and third authors will explore a more general strategy for  establishing minimum degree thresholds for perfect tilings in edge-ordered graphs.
\end{remark}

\section{Concluding remarks}\label{sec:conc}
In this paper we have characterized those edge-ordered graphs that are tileable; similarly to the characterization of Tur\'anable edge-ordered graphs, 
the tileable edge-ordered graphs $F$ are those that can be embedded in specific orderings  -- which we call the \sco{s} -- of the complete graph $K_{|F|}$. For the characterization of Tur\'anable graphs, namely Theorem~\ref{thm:turanable}, all four canonical orderings are necessary in the following sense: 
for every~$n\geq 4$ and every canonical ordering $K_n ^{\scaleto{{\leq}}{4.3pt}}$ of $K_n$, there is a non-Tur\'anable edge-ordered $n$-vertex graph $F$ such that $F$ can be embedded into all the canonical orderings of $K_n$ other than $K_n ^{\scaleto{{\leq}}{4.3pt}}$. Thus, it is natural to raise the following question.

\begin{question}\label{quest:all20needed}
    Are all twenty \sco{}s  necessary in Theorem~\ref{thm:character}? 
    That is, does  Theorem~\ref{thm:character} still hold if we
    omit some of the \sco{}s from the statement?
\end{question}

From a computer-assisted check, 
we know that at least the following \emph{eight} \sco{s} are necessary: 
smaller increasing/decreasing of types min/inverse min, and 
larger increasing/decreasing of types max/inverse max. Note that these include the four canonical orderings.


In this paper we have also answered Question~\ref{ques1} in the case of monotone paths and for a few other special types of edge-ordered graphs. 
Recall that in Section~\ref{subsec:almost} we computed the minimum degree threshold for an edge-ordered graph to contain an almost perfect $\Pk$-tiling. It is also natural to consider this problem more generally. This motivates the following definition.
\begin{definition}[Almost tileable]\label{def:almosttile}
An edge-ordered graph $F$ is \emph{almost tileable} if for every $0<\eps <1$ there exists a $t\in \mathbb N$ such that every edge-ordering of the graph $K_t$ contains an $F$-tiling covering all but at most $\eps t$ vertices of $K_t$.
\end{definition}
It is easy to see that this notion is equivalent to being Tur\'anable.
\begin{prop}\label{prop1}
An edge-ordered graph $F$ is almost tileable if and only if $F$ is Tur\'anable.
\end{prop}
\proof
The forwards direction is immediate. For the reverse direction, consider any $F$ that is Tur\'anable. Let $T$ denote the smallest integer such that every edge-ordering of $K_T$ contains a copy of $F$.
Given any $0<\eps <1$ define $t:=\lceil T/\eps \rceil$. Then given any edge-ordering of $K_t$, by definition of $T$ we may repeatedly find vertex-disjoint copies of $F$ in $K_t$ until we have covered all but fewer than $T$ vertices in $K_t$.
That is, we have an $F$-tiling covering all but at most $\eps t$ vertices of $K_t$, as desired.
\endproof
In light of Proposition~\ref{prop1} we propose the following question.
\begin{question}\label{ques2}
Let $F$ be a fixed Tur\'anable edge-ordered graph. What is the minimum degree threshold for forcing an almost perfect $F$-tiling in an edge-ordered graph on $n$ vertices?
More precisely, given any $\eps >0$, what is the minimum degree required in an $n$-vertex edge-ordered graph $G$ to force an $F$-tiling in $G$ covering all but at most $\eps n$ vertices?
\end{question}
We emphasize that just because the notions of Tur\'anable and almost tileable are equivalent, this certainly does not mean that the answer to Question~\ref{ques2} will be the `same' as the Tur\'an threshold. For example, whilst R\"odl~\cite{rodl} showed that one only requires $k(k+1)n/2$ edges in an $n$-vertex edge-ordered graph $G$ to force a copy of $\Pk$, Lemma~\ref{lemma:almosttiling} implies $G$ must be much denser to contain an almost perfect $\Pk$-tiling.

\smallskip

Perhaps one of the main open problems in the area is to
 characterize the possible Tur\'an numbers of edge-ordered graphs. 
\begin{question}\label{quesT}
For which $\alpha \geq 0$ does there exist a  Tur\'anable edge-ordered graph $F$ so that
$(\alpha +o(1))\binom{n}{2}$ is the 
Tur\'an threshold for an $n$-vertex edge-ordered graph to contain a copy of $F$?
\end{question}
Similarly to  the (unordered) graph setting, Theorem 2.3 in~\cite{bkwpstw} implies that the only $\alpha\geq 0$ that 
could be a Tur\'an number for an edge-ordered graph $F$ are of the form $\alpha =(k-1)/k$ where $k \in \mathbb N$. Thus, in Question~\ref{quesT} we seek the values of $k \in \mathbb N$ for which 
$(k-1)/k$ is a Tur\'an number of an edge-ordered graph. 
Due to \cite[Theorem 2.3]{bkwpstw}, this is in turn equivalent to asking for which $k \in \mathbb N$
 does there exist an edge-ordered graph of \emph{order chromatic number $k$}; see \cite{bkwpstw} for the definition of order chromatic number.

The following  is the tileable analog of Question~\ref{quesT}.
\begin{question}\label{quesT2}
For which $\alpha \geq 0$ does there exist a  tileable edge-ordered graph $F$ without isolated vertices so that
$f(n,F)=(\alpha +o(1))n$ for all $n$ divisible by $|F|$?
\end{question}
Proposition~\ref{propf} implies that one can take $\alpha $ equal to $1/2$ or $2/3$ here. We suspect a full resolution of Question~\ref{quesT2} will be very challenging, so it would be interesting to first establish if there are infinitely many choices for $\alpha$ in Question~\ref{quesT2}.

\smallskip

Recall that every Tur\'anable  edge-ordered graph  $F$ does not contain a copy of $K_4$. In the  version of this paper we first submitted, we asked whether it is true that for every~$k\in \mathbb N$ there is a Tur\'anable edge-ordered graph~$F$ whose underlying graph has chromatic number at least~$k$.
However, G\'abor Tardos showed us the following argument that implies Tur\'anable edge-ordered graphs must have small chromatic number.

\begin{prop}
    If $F$ is a Tur\'anable edge-ordered graph then its underlying graph has chromatic number at most $4$.
\end{prop}
\begin{proof}
    Let $n:=|F|$. As  $F$ is  Tur\'anable, there is a copy of $F$ in the min ordering of $K_n$.
    This implies that there is an ordering $v_1,\dots, v_n$ of $V(F)$ so that 
    \begin{align}
        &\text{if  $v_iv_j$, $v_kv_{\ell}$ are
    edges in $F$ so that $i<j$ and $k < \ell$, and where
    $v_iv_j <v_k v_{\ell}$, then $i \leq k$;}\label{33} \\ &\text{further, if $i=k$ then $j < \ell$.}
    \label{11}
    \end{align}
    Similarly, as $F$ is  Tur\'anable, there is a copy of $F$ in the inverse min ordering of $K_n$.
     This implies that there is an ordering $w_1,\dots, w_n$ of $V(F)$ so that
      \begin{align}
        &\text{if  $w_iw_j$, $w_kw_{\ell}$ are
    edges in $F$ so that $i<j$ and $k < \ell$, and where
    $w_iw_j <w_k w_{\ell}$, then $i \leq k$;}\label{44} \\ &\text{further, if $i=k$ then $j > \ell$.}
    \label{22}
    \end{align}

    Let $S$ denote the set of vertices $v_i$ in $F$ for which there is no $j>i$ such that $v_iv_j\in E(F)$.
    Similarly, let $S'$ denote the set of vertices $w_i$ in $F$ for which there is no $j>i$ such that $w_iw_j\in E(F)$. Clearly both $S$ and $S'$ form independent sets in $F$.

    Set $T:=V(F)\setminus (S\cup S')$. So each vertex $v_i \in T$ sends  out an edge in $F$ to a vertex $v_j$ where $j>i$. Furthermore, by (\ref{33}), the largest edge in $F$ incident to $v_i$ must be of the form $v_iv_j$ where $j>i$. Similarly,   each vertex $w_i \in T$ sends  out an edge in $F$ to a vertex $w_j$ where $j>i$. Furthermore, by (\ref{44}), the largest edge in $F$ incident to $w_i$ must be of the form $w_iw_j$ where $j>i$. These  properties  imply the following claim.
    \begin{claim}\label{claimx}
    Let $x,y \in T$ be distinct,  let   $e_1$ denote the largest edge in $F$ incident to $x$ and let $e_2 $ denote the
    largest edge in $F$ incident to $y$. If $e_1<e_2$ then $x$ is before $y$ in the ordering 
    $v_1,\dots, v_n$ of $V(F)$ {and} $x$ is before $y$ in the  ordering 
    $w_1,\dots, w_n$ of $V(F)$.
    \end{claim}
    \begin{claimproof}
       As $x \in T$, $x=v_i$ for some $i$ and $e_1=v_iv_j$ where $j>i$. Similarly, as $y \in T$,
       $y=v_k$ for some $k$ and $e_2=v_kv_\ell$ where $\ell>k$. As $v_iv_j< v_k v_{\ell}$,
       (\ref{33}) implies that $i \leq k$. Further, $v_i=x \not =y =v_k$, so in fact $i<k$. That is,
       $x$ is before $y$ in the ordering  $v_1,\dots, v_n$ of $V(F)$.
The second part of the claim follows analogously.
    \end{claimproof}
    
    Notice Claim~\ref{claimx} implies that the ordering of $T$ induced by $v_1,\dots, v_n$ 
    is the \emph{same} as the ordering of $T$ induced by $w_1,\dots, w_n$. We write this ordering of 
    $T$ as $u_1,\dots, u_t$ where $t:=|T|$.
    \begin{claim}\label{claimy}
    In $F$, every $u_i \in T$ is adjacent to at most one vertex $u_j \in T$ where $j>i$.
    \end{claim}
    \begin{claimproof}
        Suppose for a contradiction there  exist distinct $j,\ell>i$ so that both $u_iu_j$ and $u_iu_\ell$ are edges in $F[T]$. Without loss of generality suppose that $u_iu_j<u_iu_\ell$ in the total order 
        on $E(F)$. Then since the ordering $u_1,\dots, u_t$ of $T$ is induced by the ordering 
        $v_1,\dots, v_n$  of $V(F)$, (\ref{11}) implies that $j <\ell$. On the other hand,  
        the ordering $u_1,\dots, u_t$ of $T$ is also induced by the ordering 
        $w_1,\dots, w_n$  of $V(F)$. Thus, (\ref{22}) implies that $j >\ell$, a contradiction.
    \end{claimproof}
    
Claim~\ref{claimy} implies that $F[T]$ contains no cycles; so the underlying graph of $F[T]$ is
bipartite. As $V(F)=S\cup S'\cup T$ where $S$ and $S'$ are independent sets, this immediately implies
that the underlying graph of $F$ has chromatic number at most $4$, as desired.
\end{proof}
Recall the edge-ordered version of $K_3$ has chromatic number $3$ and is Tur\'anable. It would be interesting to determine whether there is a Tur\'anable edge-ordered graph of chromatic number $4$.

\subsection*{Acknowledgments}
Much of the research in this paper was carried out during a visit by the second and third authors to the University of Illinois at Urbana-Champaign. The authors are grateful to the BRIDGE strategic alliance between the University of Birmingham and the University of Illinois at Urbana-Champaign, which partially funded this visit. 
The authors are also grateful to J\'ozsef Balogh for helpful discussions,   D\"om\"ot\"or P\'alv\"olgyi and G\'abor Tardos for helpful comments, and to the referees for their careful reviews.

\smallskip

{\noindent \bf Open access statement.}
This research was funded in part by  EPSRC grant EP/V002279/1. For the purpose of open access, a CC BY public copyright licence is applied to any Author Accepted Manuscript arising from this submission.

\smallskip

{\noindent \bf Data availability statement.}
The files required for the computer-assisted check described in Section~\ref{sec:conc} can be found on the following web-page: \url{https://sipiga.github.io/Edge-Ordered_files.zip}.

\end{document}